\documentclass[12pt,oneside,a4paper,reqno]{amsart}

\usepackage{amssymb}
\usepackage{bbm}
\usepackage{cite}
\usepackage{hyperref}
\usepackage{lmodern} 

\title[]{MIT bag in non-smooth convex domains}

\author[]{Konstantin Pankrashkin}

\address{\tiny Carl von Ossietzky Universit\"at Oldenburg, Fakult\"at V -- Mathematik und Naturwissenschaften, Institut f\"ur Mathematik, Ammerl\"ander Heerstr. 114--118, 26129 Oldenburg, Germany}

\email{konstantin.pankrashkin@uol.de}

\dedicatory{Dedicated to the memory of my doctoral co-advisor Jochen Br\"uning (1947--2025)}



\urladdr{http://uol.de/pankrashkin}

\usepackage{url} 

\allowdisplaybreaks[4] 

\newcommand{\one}{\mathbbm{1}}
\newcommand{\cH}{\mathcal{H}}

\newcommand{\RR}{\mathbb{R}}
\newcommand{\CC}{\mathbb{C}}
\newcommand{\NN}{\mathbb{N}}
\newcommand{\cD}{\mathcal{D}}
\newcommand{\cE}{\mathcal{E}}
\newcommand{\cP}{\mathcal{P}}
\newcommand{\SSS}{\mathbb{S}}
\newcommand{\eps}{\varepsilon}
\newcommand{\rmi}{\mathrm{i}}
\newcommand{\cc}{\complement}
\DeclareMathOperator{\supp}{\mathrm{supp}}
\DeclareMathOperator{\dist}{\mathrm{dist}}
\DeclareMathOperator{\ran}{\mathrm{ran}}
\DeclareMathOperator{\dom}{\mathrm{dom}}

\newtheorem{theorem}{Theorem}
\newtheorem{corollary}[theorem]{Corollary}
\newtheorem{lemma}[theorem]{Lemma}
\theoremstyle{definition}
\newtheorem{definition}[theorem]{Definition}
\newtheorem{remark}[theorem]{Remark}

\renewcommand{\Tilde}{\widetilde}
\renewcommand{\Hat}{\widehat}

\newcommand{\dd}{\mathrm{d}}

\begin{document}

\begin{abstract}
The Dirac operator with MIT bag boundary condition in a bounded convex domain
is shown to be always self-adjoint in the $H^1$-setting. This allows one to show that such operators appear as limit of Dirac operators with large positive mass outside the domain. Similar results were previously known for smooth domains only.
\end{abstract}

\maketitle


\section*{Introduction}

For $n\ge 2$ and $N:=2^{[\frac{n+1}{2}]}$ let $\alpha_1,\dots,\alpha_{n},\beta$ be pairwise anticommuting Hermitian $N\times N$ matrices with $\alpha_k^2=\beta^2=I_N$, where $I_N$ is the $N\times N$ identity matrix (see Appendix~\ref{appa} for a possible explicit choice). The Euclidean Dirac operator $D_m$ with a mass $m\in\RR$ acts on vector functions $f:\RR^n\to \CC^N$ by the differential expression
\begin{equation}
	\label{eqedm}
	D_m f:=-\rmi\sum_{k=1}^n \alpha_k \partial_k f + m \beta f,
\end{equation}
and it is a central object in the relativistic quantum mechanics, see e.g. \cite{thaller}.

Now let $\Omega\subset\RR^n$ be a bounded domain with sufficiently regular (e.g. Lipschitz) boundary and outer unit normal $\nu$. By the MIT bag operator
with mass $m$ in $\Omega$ one means the operator $A^\Omega_m$ acting in $L^2(\Omega,\CC^N)$ as $f\mapsto D_m f$
on the functions $f$ satisfying the MIT bag boundary condition $f=-\rmi\beta\alpha\cdot\nu f$ on $\partial\Omega$, where
we use the writing
\[
\alpha\cdot x=\sum_{k=1}^n x_k \alpha_k \text{ for any }x=(x_1,\dots,x_n)\in\RR^n
\]
and denote a function on $\Omega$ and its trace on $\partial\Omega$ by the same symbol for the sake of readability.
The precise regularity of the functions $f$ in the operator domain guaranteeing ``good'' properties of $A^\Omega_m$ will be discussed below in greater detail.

The Dirac operator with the above boundary condition appeared initially in the physics literature \cite{bogol,chod,john}, in particular, as a model of quark confinement in hadrons. One of the interesting features observed is that the MIT bag operator appears as the limit of Dirac operators in the whole space with a large mass applied outside $\Omega$ \cite{BM}, which explains the fact that the MIT bag boundary condition is often referred to as the ``infinite mass boundary condition'' (and it is usually considered as an analog of the Dirichlet boundary condition for the Laplacian). While the study in the physics literature has a long story, it seems that the mathematically rigorous analysis of the operator $A^\Omega_m$ only became active in the last decade. For $\Omega$ with smooth boundary and $n\in\{2,3\}$ several authors \cite{ALTR,BHM,beng, OBV} have shown
that the above operator $A^\Omega_m$ is self-adjoint if considered on the  domain
\begin{equation}
	\label{doma00}
\big\{ f\in H^1(\Omega,\CC^N):\ f=-\rmi \beta \alpha\cdot \nu f \text{ on } \partial\Omega\big\}.
\end{equation}
As noted in \cite{mobp}, the result holds in any dimension as a corollary of more general results on manifolds with boundaries \cite{baer,baer2}. Moreover, if one denotes by $E_j(T)$ the $j$-th eigenvalue of a lower semibounded operator $T$ (assuming the usual numbering in the non-decreasing order with multiplicities counted), for each $j\in\NN$
one has
\begin{equation}
	\label{abm}
E_j\big((A^\Omega_m)^2\big)=\lim_{M\to+\infty} E_j\big((B^\Omega_{m,M})^2\big),
\end{equation}
where $B^\Omega_{m,M}$ is the Dirac operator in $L^2(\RR^n,\CC^N)$ with mass $m$ in $\Omega$ and mass $M$ in $\Omega^\cc$,
\begin{equation}
	\label{bmm}
\begin{aligned}
	B^\Omega_{m,M}:&\ f\mapsto D_0 f +  (m \one_{\Omega} + M \one_{\Omega^\cc})\beta f \equiv D_m f+(M-m)\one_{\Omega^\cc}\beta f,\\
	&\dom B^\Omega_{m,M}=H^1(\RR^n,\CC^N),
\end{aligned}
\end{equation}
see \cite{SV} for $n=2$, or \cite{ALTMR} for $n=3$, or \cite{mobp} for arbitrary $n$.
The relation \eqref{abm} is usually viewed as the mathematical expression of the infinite mass limit mentioned above.
Note that there are some extensions to unbounded domains \cite{BCLTS} and spin manifolds \cite{flam},
and some estimates for the rate of convergence are available as well \cite{ALTMR,BBZ}.
It should be noted that the proofs of the convergence are mainly based on the expression
\begin{equation}
	\label{qf-am}
\|A^\Omega_m f\|^2_{L^2(\Omega,\CC^N)}=\int_\Omega \Big( |\nabla f|^2 +m^2|f|^2\Big)\,\dd x+\int_{\partial\Omega}\Big(m+\dfrac{H}{2}\Big) |f|^2\,\dd \cH^{n-1}, \quad f\in\dom A^\Omega_m,
\end{equation}
where $H: x\mapsto \mathop{\mathrm{tr}}\dd|_x\nu$ is the mean curvature of $\partial\Omega$ and $\cH^{n-1}$ is the $(n-1)$-dimensional Hausdorff measure, and on the use of tubular coordinates near $\partial\Omega$.
Remark that the same expression also plays a central role in the spectral optimization \cite{ANT}.

While there is a vast literature on boundary value problems for Dirac-type operators on domains/manifolds with smooth boundaries, see e.g.~\cite{baer,baer2,bl,bbb,grosse} and references therein, the case of non-smooth boundaries
is much less elaborated. The paper \cite{LTOB} studied the self-adjointness of $A^\Omega_m$ for the case $n=2$ when $\Omega$ is a polygon, and it was shown that the domain \eqref{doma00} leads to a self-adjoint operator if and only if the polygon is convex (more generally, the deficiency indices conicide with the number of concave corners). In \cite{CL} it was shown that $A^\Omega_m$ is self-adjoint on the domain \eqref{doma00} if $\Omega$ is a convex circular cone in three dimensions. 
As shown in \cite{BB1,BHSS,PV}, for bounded Lipschitz $\Omega$ and $n\in\{2,3\}$ the operator $A^\Omega_m$ becomes self-adjoint
if considered on the larger domain
\[
\big\{ f\in H^{\frac{1}{2}}(\Omega,\CC^N): \  D_m f\in L^2(\Omega,\CC^N),\ f=-\rmi \beta \alpha\cdot \nu f \text{ on } \partial\Omega\big\},
\]
and the paper \cite{vu} studied the local regularity of eigenfunctions. It should be noted that the lost of regularity
also destroys the proofs for \eqref{abm}, as the finiteness of the summands in the crucially important identity \eqref{qf-am}
implicitly means the inclusion $\dom A^\Omega_m\subset H^1(\Omega,\CC^N)$. The paper \cite{BBZ} contains an alternative proof of \eqref{abm} for $n=3$ and smooth $\Omega$ with the help of the resolvents, but its main constructions are based on the microlocal analysis, and no obvious extension to non-smooth domains is expected. Overall, to our knowledge, there are no proofs of the asymptotics \eqref{abm} for any non-smooth $\Omega$.

The main goal of the present paper is to show that the above results for smooth $\Omega$ can be extended to a large class of non-smooth but \emph{convex} $\Omega$. Remark that the convexity is one of the standard assumptions guaranteeing maximal regularity for Laplacians with various boundary conditions, see e.g. \cite[Chap.~3]{gris} for most classical results or \cite{lp1,lp2} for recent developments. Before passing to rigorous formulations we recall that any bounded convex domain has Lipschitz boundary \cite[Corollary 1.2.2.3]{gris}, hence, the usual machineries in Sobolev spaces (embeddings, traces, extension operators) are available.
Our central result is as follows:
\begin{theorem}[Self-adjointness]\label{thm-selfadj}
	Let $\Omega\subset\RR^n$ be a bounded convex domain with outer unit normal $\nu$ \textup{(}which is defined a.e. on $\partial\Omega$\textup{)}, and let $m\in\RR$. Then the operator
	\begin{equation}
		\label{doma}
		A_m^\Omega:\ f\mapsto D_m f,\quad \dom A^\Omega_m=\big\{ f\in H^1(\Omega,\CC^N):\ f=-\rmi \beta \alpha\cdot \nu f \text{ on } \partial\Omega\big\},
	\end{equation}
	is self-adjoint in $L^2(\Omega,\CC^N)$ with compact resolvent.
\end{theorem}

For the subsequent discussion we impose a better regularity of the boundary.

\begin{definition}\label{def-pcs}
	Let $U\subset\RR^n$ be a bounded domain with Lipschitz boundary. A point $s\in \partial U$ is called \emph{regular} if $U$ coincides with a $C^\infty$-smooth domain in an open neighborhood of $s$, otherwise it is called \emph{singular}. One denotes
	\begin{align*}
	\partial_\infty U&:=\{s\in \partial U:\ \text{$s$ is regular}\} &&\text{ (the \emph{regular boundary} of $U$),}\\
	\partial_0 U&:=\{s\in \partial U:\ \text{$s$ is singular}\} &&\text{ (the \emph{singular boundary} of $U$).}
	\end{align*}
	The domain $U$ will be called \emph{piecewise smooth} if its singular boundary is contained in the union of finitely many compact $(n-2)$-dimensional submanifolds of $\RR^n$.	
\end{definition}	

We show:

\begin{theorem}[Quadratic form]\label{thm2}
	Let $\Omega\subset\RR^n$ be a bounded convex domain with piecewise smooth boundary and outer unit normal $\nu$ \textup{(}which is defined at least on the regular boundary\textup{)}, then
	for any $f\in \dom A^\Omega_m$ one has
	\[
	\|A^\Omega_m f\|^2_{L^2(\Omega,\CC^N)}=\int_\Omega \Big( |\nabla f|^2 +m^2|f|^2\big)\,\dd x+\int_{\partial\Omega}\Big(m+\dfrac{H}{2}\Big) |f|^2\dd \cH^{n-1},
	\]
	where $H: x\mapsto \mathop{\mathrm{tr}}\dd|_x\nu\ge 0$ is the mean curvature defined on the regular boundary of $\Omega$.	
\end{theorem}

By imposing an additional geometric condition we are finally able to extend \eqref{abm} to a large class of non-smooth domains:

\begin{theorem}[Infinite mass limit]\label{thm-limit}
	Let $\Omega\subset\RR^n$ be a bounded convex domain with piecewise smooth boundary, and assume that the mean curvature of the regular boundary is bounded. Then for any $m\in\RR$ and $j\in\NN$ one has
	$E_j\big((A^\Omega_m)^2\big)=\lim\limits_{M\to+\infty} E_j\big((B^\Omega_{m,M})^2\big)$,
	where the operator $B^\Omega_{m,M}$ acting in $L^2(\RR^n,\CC^N)$ is defined by \eqref{bmm}.
\end{theorem}

The assumptions of Theorem~\ref{thm-limit} are satisfied, in particular, if $\Omega$ is a bounded convex polyhedron (i.e. a bounded domain obtained as an intersection of halfspaces), or a convex curvilinear polyhedron (i.e. obtained as a diffeomorphic image of a bounded convex polyhedron), or is a convex piecewise smooth domain obtained as the intersection of finitely many $C^\infty$-smooth bounded domains. On the other hand, the assumption on the boundedness of the mean curvature
excludes e.g. the three-dimensional domains whose boundary has conical singularities of the form $x_3=\sqrt{x_1^2+x_2^2}$. We believe that the assumption on the mean curvature is purely technical, see Remark~\ref{rmk-end} at the end of the paper.

Our proof of Theorem~\ref{thm-selfadj} is presented in Section~\ref{proof-thm1}, and it is based on adaptations of some constructions in Grisvard's book \cite[Sec.~3.2]{gris} for second-order operators. Remark~\ref{rmk10} indicates possible extensions. In Section~\ref{approx}
we derive some approximation results for Sobolev spaces. This allows one to show the essential self-adjoitness of $A^\Omega_m$ on smooth functions vanishing near the singular boundary (Lemma~\ref{thm-ess}), which is an essential ingredient for the subsequent analysis. In particular, it is used in Section~\ref{proof-thm2} to prove Theorem~\ref{thm2}
by transferring the results known for the smooth case. Theorem~\ref{thm-limit} is shown in Section~\ref{proof-thm3},
and is mainly in the spirit of the analysis of~\cite{mobp}, while a more involved choice of test functions is used, and the constructions with tubular coordinates near $\partial\Omega$ (which are now unavailable due to non-smoothness) are replaced
by some estimates for Robin Laplacians in exterior domains. The short Appendix~\ref{appa} collects a necessary information on the Dirac matrices.

\section{Proof of Theorem \ref{thm-selfadj} (self-adjointness)}\label{proof-thm1}

As a starting point let us summarize rigorously some results on  smooth domains, see \cite[Lem.~2.1 and Prop.~A.2]{mobp}.
\begin{lemma}\label{lem2}
	Let $\Omega\subset\RR^n$ be a bounded domain with $C^\infty$-smooth boundary and outer unit normal $\nu$, and let $m\in\RR$. Then the operator $A^\Omega_m$ defined by \eqref{doma}
	is self-adjoint in $L^2(\Omega,\CC^N)$, and for any $f\in \dom A^\Omega_m$ one has the identity \eqref{qf-am}.
\end{lemma}

As the parameter $m$ only results in adding a bounded symmetric perturbation, it is sufficient to show that
the operator
\[
A:=A^\Omega_0
\]
is self-adjoint and has compact resolvent. The integration by parts show that $A$ is symmetric, and its self-adjointness is equivalent to
the equalities $\ran(A\pm\rmi)=L^2(\Omega,\CC^N)$. We will show $\ran(A+\rmi )=L^2(\Omega,\CC^N)$, as the other condition
is proved completely analogously by adjusting the signs.

Our constructions will be based on the following results of convex analysis, see e.g. \cite[Sec.~3.4]{schneider} and \cite[Lemma 3.2.3.2]{gris}:

\begin{lemma}\label{lem1}
	Let $\Omega\subset\RR^n$ be a bounded convex domain. There exist bounded convex $C^\infty$-smooth domains $\Omega_p\subset\RR^n$, $p\in\NN$, such that $\Omega\subset \Omega_p$ for all $p$ and $d_H(\partial \Omega_p,\partial \Omega)\xrightarrow{p\to\infty}0$, where $d_H$ stands for the Hausdorff distance. In addition, one can find open subsets $V_k\subset\RR^n$, $k\in\{1,\dots,K\}$,
	with the following properties\textup{:}
	\begin{itemize}
		\item for each $k$ there exist Cartesian coordinates $y^k_1,\dots,y^k_n$ in which $V_k$ is a hypercube,
		\[
		V_k=\Big\{(y^k_1,\dots,y^k_n): -a^k_j<y^k_j<a^k_j \text{ for all } j\in\{1,\dots,n\}\Big\},\quad a^k_j>0,
		\]
		\item for all $p\in\NN$ and each $k\in\{1,\dots,K\}$ there exist a Lipschitz function $h^k$ and a $C^\infty$-smooth function $h^k_p$ defined on
		\[
		V'_k=\Big\{(y^k_1,\dots,y^k_{n-1}): -a^k_j<y^k_j<a^k_j \text{ for all } j\in\{1,\dots,n-1\}\Big\}
		\]
		such that
		\begin{align*}
			& h^k(z^k)\le h^k_p(z^k),\ 			
			\big|h^k(z^k)\big|\le \frac{a^k_n}{2},\ \big|h^k_p(z^k)\big|\le \frac{a^k_n}{2} \text{ for all }z^k\in V'_k,\\
			\Omega\cap V_k&=\big\{ y^k=(z^k,y^k_n)\in V_k:\, z^k\in V'_k,\ y^k_n< h^k(z^k)\big\}, \\
			\Omega_p\cap V_k&=\big\{ y^k=(z^k,y^k_n)\in V_k:\, z^k\in V'_k,\ y^k_n< h^k_p(z^k)\big\},\\
			\partial \Omega\cap V_k&=\big\{ y^k=(z^k,y^k_n)\in V_k:\, z^k\in V'_k,\ y^k_n= h^k(z^k)\big\},\\
			\partial \Omega_p\cap V_k&=\big\{ y^k=(z^k,y^k_n)\in V_k:\, z^k\in V'_k,\ y^k_n= h^k_p(z^k)\big\},
		\end{align*}
		\item $\partial \Omega\subset\bigcup_{k=1}^K V_k$ and $\partial \Omega_p\subset\bigcup_{k=1}^K V_k$ for all $p\in\NN$,
		\item for each $k\in\{1,\dots,K\}$ one has $h^k_p\xrightarrow{p\to\infty} h^k$ uniformly in $V'_k$,
		\item there is an $L>0$ such that $	\big|\nabla h^k(z^k)\big|\le L$ and $\big|\nabla h^k_p(z^k)\big|\le L$
		for all a.e. $z^k\in V'_k$, all $k\in\{1,\dots,K\}$ and all $p\in\NN$,
		\item for each $k\in\{1,\dots,K\}$ one has $\nabla h^k_p(z^k)\xrightarrow{p\to\infty} \nabla h^k(z^k)$ for a.e. $z^k\in V'_k$.		
	\end{itemize}
\end{lemma}

Let us approximate $\Omega$ by bounded convex $C^\infty$-smooth domains $\Omega_p$ (with $p\in\NN$) as in Lemma~\ref{lem1} and consider
the associated Dirac operators
\[
\Tilde A_p:=A^{\Omega_p}_0,
\]
which are self-adjoint in $L^2(\Omega_p,\CC^N)$ by Lemma~\ref{lem2}.

Let $g\in L^2(\Omega,\CC^N)$. For each $p$ denote by $\Tilde g_p$ the continuation of $g$ by zero to $\Omega_p$ then due to the self-adjointness of $\Tilde A_p$ there is a unique $\Tilde f_p\in \dom \Tilde A_p$ with $(\Tilde A_p+\rmi)\Tilde f_p=\Tilde g_p$. Let $H_p$ be the mean curvature of $\partial\Omega_p$,
then $H_p\ge 0$ due to convexity of $\Omega_p$, and by Lemma~\ref{lem2} we obtain
\begin{equation}
	\label{ggff}
\begin{aligned}
\|g\|^2_{L^2(\Omega,\CC^N)}&=\|\Tilde g_p\|^2_{L^2(\Omega_p,\CC^N)}=\big\|(\Tilde A_p+\rmi)\Tilde f_p\big\|^2_{L^2(\Omega_p,\CC^N)}= \|\Tilde A_p \Tilde f_p\|^2_{L^2(\Omega_p,\CC^N)}+ \|\Tilde f_p\|^2_{L^2(\Omega_p,\CC^N)}\\
&=\int_{\Omega_p} \big(|\nabla \Tilde f_p|^2 +|\Tilde f_p|^2\big)\,\dd x+\frac{1}{2}\int_{\partial\Omega_p} H_p |\Tilde f_p|^2\dd \cH^{n-1}\\
& \ge \int_{\Omega_p} \big(|\nabla \Tilde f_p|^2 +|\Tilde f_p|^2\big)\,\dd x=\|\Tilde f_p\|^2_{H^1(\Omega_p,\CC^N)}.
\end{aligned}
\end{equation}
Denote by $f_p$ the restriction of $\Tilde f_p$ on $\Omega$, then
\begin{equation}
	\label{normineq-fp}
\|f_p\|_{H^1(\Omega,\CC^N)}\le \|\Tilde f_p\|_{H^1(\Omega_p,\CC^N)}\le \|g\|_{L^2(\Omega,\CC^N)},
\end{equation}
which shows that the sequence $(f_p)_{p\in\NN}$ is bounded in $H^1(\Omega,\CC^N)$ and, hence, contains a weakly convergent subsequence. To keep simple notation we assume (by replacing the initial sequence of $\Omega_p$ by the corresponding subsequence)
that $f_p\xrightarrow{p\to\infty}f$ weakly in $H^1(\Omega,\CC^N)$
for some $f\in H^1(\Omega,\CC^N)$. We are going to show that $f\in \dom A$ with $(A+\rmi)f=g$, which will conclude the proof
of self-adjointness for $A$.

Due to the compactness of the embedding $H^1(\Omega,\CC^N)\hookrightarrow L^2(\Omega,\CC^N)$ we have
\begin{equation}
\label{fpf}
 \|f_p-f\|_{L^2(\Omega,\CC^N)}\xrightarrow{p\to\infty}0.
\end{equation}
Let $\varphi\in C^\infty_c(\Omega,\CC^N)$ and $\Tilde\varphi_p$ be its extension by zero on $\Omega_p$. Then $\Tilde \varphi_p\in C^\infty_c(\Omega_p,\CC^N)\subset \dom \Tilde A_p$, and
\begin{align*}
\langle g,\varphi\rangle_{L^2(\Omega,\CC^N)}&=
\langle \Tilde g_p,\Tilde \varphi_p\rangle_{L^2(\Omega_p,\CC^N)}=
\big\langle (\Tilde A_p+\rmi)\Tilde f_p,\Tilde \varphi_p\big\rangle_{L^2(\Omega_p,\CC^N)}\\
&=\big\langle \Tilde f_p,(\Tilde A_p-\rmi)\Tilde \varphi_p\big\rangle_{L^2(\Omega_p,\CC^N)}
=\big\langle \Tilde f_p,(D_0-\rmi)\Tilde \varphi_p\big\rangle_{L^2(\Omega_p,\CC^N)}\\
&=\big\langle  f_p,(D_0-\rmi) \varphi\big\rangle_{L^2(\Omega,\CC^N)}\xrightarrow[\text{by \eqref{fpf}}]{p\to\infty}
\big\langle  f,(D_0-\rmi) \varphi\big\rangle_{L^2(\Omega,\CC^N)}.
\end{align*}
Hence,
\[
\int_\Omega \big\langle g(x),\varphi(x)\big\rangle_{\CC^N}\dd x=
\int_\Omega \big\langle f(x),(D_0-\rmi)\varphi(x)\big\rangle_{\CC^N}\dd x
\text{ for all }\varphi\in C^\infty_c(\Omega,\CC^N),
\]
which means $(D_0+\rmi)f=g$ in $\cD'(\Omega,\CC^N)$. As $f\in H^1(\Omega,\CC^N)$ by construction, it remains to check that $f$ satisfies the required boundary condition.

Let $V_k$, $V'_k$, $h^k$, $h^k_p$ and $L$ be as in Lemma~\ref{lem1}, and let $\nu^k(z^k)$ be the outer unit normal for $\Omega$ at the point $\big(z^k, h^k(z^k)\big)\in V_k\cap\partial\Omega$
and $\nu^k_p(z^k)$ be the outer unit normal for $\Omega_p$ at the point $\big(z^k, h^k_p(z^k)\big)\in V_k\cap \partial\Omega_p$, 
i.e.
\begin{align*}
	\nu^k(z^k)&=\dfrac{\big(-\nabla h^k(z^k),1\big)}{\sqrt{1+\big|\nabla h^k(z^k)\big|^2}},&
	\nu^k_p(z^k)&=\dfrac{\big(-\nabla h^k_p(z^k),1\big)}{\sqrt{1+\big|\nabla h^k_p(z^k)\big|^2}}.
\end{align*}
In order to show that the function $f$ satisfies the boundary condition we need to check that
for any $k\in\{1,\dots,K\}$ one has
\begin{multline*}
\int_{V_k\cap\partial\Omega} \big| (1+\rmi \beta\alpha\cdot \nu) f\big|^2\dd\cH^{n-1}\equiv
\int_{V'_k} \Big| \big(1+\rmi \beta\alpha\cdot \nu^k(z^k)\big) f\big(z^k,h^k(z^k)\big)\Big|^2\sqrt{1+\big|\nabla h^k(z^k)\big|^2}\dd z^k=0,
\end{multline*}
and in view of the bound $\big|\nabla h^k(z^k)\big|\le L$ this is equivalent to
\begin{equation}
	\label{lok1}
\int_{V'_k} \Big| \big(1+\rmi \beta\alpha\cdot \nu^k(z^k)\big) f\big(z^k,h^k(z^k)\big)\Big|^2\dd z^k=0.
\end{equation}

By construction, the functions $\Tilde f_p$ satisfy
$(1+\rmi \beta\alpha\cdot \nu_p) \Tilde f_p=0$ on $\partial\Omega_p$, with $\nu_p$ being the unit normal for $\Omega_p$, which means that
for any for any $k\in\{1,\dots,K\}$ it holds that
\[
\int_{V_k\cap\partial\Omega_p} \big| (1+\rmi \beta\alpha\cdot \nu_p) \Tilde f_p\big|^2\dd\cH^{n-1}=0,
\]
hence,
\begin{equation}
	\label{fp-lok1}
	\int_{V'_k} \Big| \big(1+\rmi \beta\alpha\cdot \nu^k_p(z^k)\big) \Tilde f_p\big(z^k,h^k(z^k)\big)\Big|^2\dd z^k=0.
\end{equation}
Our goal is to pass to the limit $p\to\infty$ in this last formula in order to obtain the required identities \eqref{lok1}.
As a preparation let us derive  some preliminary estimates.

\begin{lemma}\label{lem4}
For any $F\in H^1(\RR^n,\CC^N)$ any any $k\in\{1,\dots,K\}$ one has
\[
\int_{V'_k} \Big| F\big(z^k,h^k_p(z^k)\big)-F\big(z^k,h^k(z^k)\big)\Big|^2\dd z^k\xrightarrow{p\to\infty}0.
\]
\end{lemma}

\begin{proof}
Using the Cauchy-Schwarz inequality we have
\begin{align*}
\int_{V'_k} \Big| F\big(z^k,h^k_p(z^k)\big)-&F\big(z^k,h^k(z^k)\big)\Big|^2\dd z^k
=
\int_{V'_k} \bigg|\int_{h^k(z^k)}^{h^k_p(z^k)} \partial_{y^k_n} F(z^k,t)\,\dd t\bigg|^2\dd z^k\\
&\le
\int_{V'_k} \big| h^k_p(z^k)-h^k(z^k)\big| \int_{h^k(z^k)}^{h^k_p(z^k)} \big|\partial_{y^k_n} F(z^k,t)\big|^2 \dd t\,\dd z^k\\
&\le \sup_{z^k\in V'_k} \big| h^k_p(z^k)-h^k(z^k)\big|\int_{V_k\cap \Omega_p} |\partial_{y^k_n} F|^2 \dd y^k\\
&\le \sup_{z^k\in V'_k} \big| h^k_p(z^k)-h^k(z^k)\big| \|F\|^2_{H^1(\RR^n,\CC^N)} \xrightarrow{p\to\infty}0
\end{align*}
due to the uniform convergence of $h^k_p$ to $h^k$ on $V'_k$.
\end{proof}

\begin{lemma}\label{lem5}
Let $F\in H^1(\RR^n,\CC^N)$ be a continuation of $f$ on $\RR^n$, then for any $k\in\{1,\dots,K\}$ one has
	\[
\int_{V'_k} \Big| F\big(z^k,h^k_p(z^k)\big)-\Tilde f_p\big(z^k,h^k_p(z^k)\big)\Big|^2\dd z^k\xrightarrow{p\to\infty}0.
	\]
\end{lemma}

\begin{proof}
We need to show that
\[
\int_{V'_k} \Big|\eta_p\big(z^k,h^k_p(z^k)\big)\Big|^2\dd z^k\xrightarrow{p\to\infty}0 \text{ for the functions }\eta_p:=F-\Tilde f_p\in H^1(\Omega_p).
\]	
First remark that due to \eqref{normineq-fp} one can find a $c>0$ such that
\begin{equation}
	\label{etap2}
\|\eta_p\|_{H^1(\Omega_p,\CC^N)}\le c \text{ for all } p\in\NN.
\end{equation}
Furthermore, by assumption $\eta_p\to0$ weakly in $H^1(\Omega)$. Using the boundedness of the trace operator $H^1(\Omega)\to H^\frac{1}{2}(\partial \Omega)$ and the compactness of the embedding $H^\frac{1}{2}(\partial \Omega)\hookrightarrow L^2(\partial \Omega)$
we obtain
\[
\int_{\partial \Omega} |\eta_p|^2\dd\cH^{n-1}\xrightarrow{p\to\infty}0,
\]
and by combining with
\begin{align*}
\int_{V'_k} \Big|\eta_p\big(z^k,h^k(z^k)\big)\Big|^2\dd z^k&\le \dfrac{1}{\sqrt{1+L^2}}
\int_{V'_k} \Big|\eta_p\big(z^k,h^k(z^k)\big)\Big|^2\sqrt{1+\big|\nabla h^k(z^k)\big|^2}\dd z^k\\
&=\dfrac{1}{\sqrt{1+L^2}}\int_{V_k\cap \partial \Omega}|\eta_p|^2\dd\cH^{n-1}
\le \dfrac{1}{\sqrt{1+L^2}}\int_{\partial \Omega}|\eta_p|^2\dd\cH^{n-1}
\end{align*}
we conclude that
\begin{equation}
	\label{etap1}
\int_{V'_k} \Big|\eta_p\big(z^k,h^k(z^k)\big)\Big|^2\dd z^k\xrightarrow{p\to\infty}0.
\end{equation}

Now we estimate
\begin{align*}
\int_{V'_k}	\Big|\eta_p\big(z^k,&h^k_p(z^k)\big)\Big|^2\dd z^k=
	\int_{V'_k}\Big|\eta_p\big(z^k,h^k(z^k)\big)+\int_{h^k(z^k)}^{h^k_p(z^k)} \partial_{y^k_n} \eta_p(z^k,t)\,\dd t\Big|^2\dd z^k\\
	&\le 2\int_{V'_k}\Big|\eta_p\big(z^k,h^k(z^k)\big)\Big|\dd z^k
	+2\int_{V'_k}\Big|\int_{h^k(z^k)}^{h^k_p(z^k)} \partial_{y^k_n} \eta_p(z^k,t)\,\dd t\Big|^2\dd z^k.
\end{align*}
For $p\to\infty$ the first summand on the right-hand side converges to zero by \eqref{etap1}, and the second summand is easily controlled using Cauchy-Schwarz inequality:
\begin{align*}
	\int_{V'_k} \bigg|\int_{h^k(z^k)}^{h^k_p(z^k)} \partial_{y_n} \eta_p(z^k,t)\dd t\bigg|^2\dd z^k
	&\le
	\int_{V'_k} \big| h^k_p(z^k)-h^k(z^k)\big| \int_{h^k(z^k)}^{h^k_p(z^k)}\big|\partial_{y^k_n} \eta_p(z^k,t)\big|^2 \dd t\,\dd z^k\\
	&\le \sup_{z^k\in V'_k} \big| h^k_p(z^k)-h^k(z^k)\big|\int_{V_k\cap \Omega_p} |\partial_{y^k_n} \eta_p|^2 \dd y^k\\
	&\le \sup_{z^k\in V'_k} \big| h^k_p(z^k)-h^k(z^k)\big| \|\eta_p\|^2_{H^1(V_k\cap \Omega_p)}\\
\text{use \eqref{etap2}: }	&\le c^2\sup_{z^k\in V'_k} \big| h^k_p(z^k)-h^k(z^k)\big| \xrightarrow{p\to\infty}0
\end{align*}
due to the uniform convergence of $h^k_p$ to $h^k$ on $V'_k$.
\end{proof}

Let $k\in\{1,\dots,K\}$ and let $F\in H^1(\RR^n,\CC^N)$ be
an arbitrary continuation of $f$, then
\begin{gather}
	\label{nabla1}
		\|\beta \alpha\cdot \nu^k(z^k)\|\le 1 \text{ and }\|\beta \alpha\cdot \nu^k_p(z^k)\|\le 1
	\text{ for all $p\in\NN$ and a.e. } z^k\in V'_k,\\
\label{nabla2}
	\beta \alpha\cdot \nu^k_p(z^k)\xrightarrow{p\to\infty}\beta \alpha\cdot \nu^k(z^k) \text{ for a.e. } z^k\in V'_k.
\end{gather}
Using the representation
\begin{align*}
	\big(1+\rmi \beta\alpha\cdot \nu^k(z^k)\big)& f\big(z^k,h^k(z^k)\big)
	=\big(1+\rmi \beta\alpha\cdot \nu^k(z^k)\big) F\big(z^k,h^k(z^k)\big)\\
	&=\big(\rmi \beta\alpha\cdot \nu^k(z^k)-\rmi \beta\alpha\cdot \nu^k_p(z^k)\big)F\big(z^k,h^k(z^k)\big)\\
	&\quad +\big(1+\rmi \beta\alpha\cdot \nu^k_p(z^k)\big)\Big( F\big(z^k,h^k(z^k)\big)-F\big(z^k,h^k_p(z^k)\big)\Big)\\
	&\quad + \big(1+\rmi \beta\alpha\cdot \nu^k_p(z^k)\big)\Big( F\big(z^k,h^k_p(z^k)\big)-\Tilde f_p\big(z^k,h^k_p(z^k)\big)\Big)\\
	&\quad +\big(1+\rmi \beta\alpha\cdot \nu^k_p(z^k)\big) \Tilde f_p\big(z^k,h^k_p(z^k)\big)\Big)
\end{align*}
we obtain 
\begin{align}
	\label{lok2}
		\int_{V'_k}\Big|\big(1+&\rmi \beta\alpha\cdot \nu^k(z^k)\big) f\big(z^k,h^k(z^k)\big)\Big|^2\dd z^k\le 4(I_p^1+I_p^2+I_p^3+I_p^4)\\
		\text{for}\quad I^1_p&:=\int_{V'_k}\Big|\big(\rmi \beta\alpha\cdot \nu^k(z^k)-\rmi \beta\alpha\cdot \nu^k_p(z^k)\big)F\big(z^k,h^k(z^k)\big)\Big|^2\dd z^k,\nonumber \\
		I^2_p&:=\int_{V'_k}\Big|\big(1+\rmi \beta\alpha\cdot \nu^k_p(z^k)\big)\Big( F\big(z^k,h^k(z^k)\big)-F\big(z^k,h^k_p(z^k)\big)\Big)
		\Big|^2\dd z^k,\nonumber\\
		I^3_p&:=\int_{V'_k}\Big|
		\big(1+\rmi \beta\alpha\cdot \nu^k_p(z^k)\big)\Big( F\big(z^k,h^k_p(z^k)\big)-\Tilde f_p\big(z^k,h^k_p(z^k)\big)\Big)
		\Big|^2\dd z^k,\nonumber\\
		I^4_p&:=\int_{V'_k}\Big|
		\big(1+\rmi \beta\alpha\cdot \nu^k_p(z^k)\big) \Tilde f_p\big(z^k,h^k_p(z^k)\big)
		\Big|^2\dd z^k\stackrel{\eqref{fp-lok1}}{=}0.\nonumber
\end{align}
In virtue of \eqref{nabla1} and \eqref{nabla2}, the dominated convergence theorem implies $I^1_p\xrightarrow{p\to\infty}0$.
Using \eqref{nabla1} again we estimate
\begin{align*}
	I^2_p&\le 4\int_{V'_k}\Big|F\big(z^k,h^k(z^k)\big)-F\big(z^k,h^k_p(z^k)\big)\Big|^2\dd z^k
	\xrightarrow{p\to\infty}0 \text{ by Lemma \ref{lem4},}\\
	I^3_p&\le 4\int_{V'_k}\Big|F\big(z^k,h^k_p(z^k)\big)-\Tilde f_p\big(z^k,h^k_p(z^k)\big)\Big|^2\dd z^k\xrightarrow{p\to\infty}0\text{ by Lemma \ref{lem5}.}
\end{align*}
This shows that the right-hand side of \eqref{lok2} converges to $0$ for $p\to\infty$, hence, the left-hand side
is zero. The proof of self-adjointness is completed.

For the sake of completeness we include the following standard assertion:
\begin{lemma}\label{lem-equiv}
	 The graph norm of $A$ on $\dom A$ is equivalent to the $H^1$-norm, and  $A$ has compact resolvent.
\end{lemma}

\begin{proof}
The self-adjointness of $A$ implies that $\dom A$ is complete with respect to the graph norm. At the same time,
$\dom A$ is also complete with respect to the $H^1$-norm being the kernel of the bounded linear operator
$H^1(\RR^n,\CC^N)\ni f\mapsto (1+\rmi\beta\alpha\cdot\nu)f|_{\partial\Omega}\in L^2(\partial\Omega)$.
The embedding operator $(\dom A \text{ with $H^1$-norm}) \ni f\mapsto f\in (\dom A \text{ with the graph norm})$
is obviously bijective and bounded, hence, its inverse is bounded due to the closed graph theorem, and this shows that the both norms are equivalent. It follows that $\dom A$ with the graph norm is continuously embedded into $H^1(\Omega,\CC^N)$ and, therefore, compactly embedded in $L^2(\Omega,\CC^N)$, which gives the conclusion.
\end{proof}

\begin{remark}\label{rmk10}
	One easily sees that the convexity was actually used only to establish a uniform upper bound for the $H^1$-norms of the functions $\Tilde f_p$ in~\eqref{ggff}. In fact, the proof can be adapted to a larger class of domains of $\Omega$ as follows.
	Let us drop the convexity assumption but require instead that $\Omega$ can be approximated (in the sense described in Lemma~\ref{lem1})
	by $C^\infty$-smooth domains $\Omega_p$ such that the mean curvatures $H_p$ of their boundaries are uniformly bounded from below, $H_p\ge -2m$ for some $m>0$. Then one can consider the operators $A:=A^\Omega_m$ and $\Tilde A_p:=A^{\Omega_p}_m$.
	If $g\in L^2(\Omega,\CC^N)$ and $\Tilde g_p$ are its extensions by zero to $\Omega_p$, then we can find $\Tilde f_p\in\dom \Tilde A_p$
	with $(\Tilde A_p+\rmi)\Tilde f_p=\Tilde g_p$, and \eqref{ggff} can be replaced by
	\begin{align*}
		\|g\|^2_{L^2(\Omega,\CC^N)}&=\|\Tilde g_p\|^2_{L^2(\Omega_p,\CC^N)}=\|(\Tilde A_p+\rmi)\Tilde f_p\|^2_{L^2(\Omega_p,\CC^N)}= \|\Tilde A_p \Tilde f_p\|^2_{L^2(\Omega_p,\CC^N)}+ \|\Tilde f_p\|^2_{L^2(\Omega_p,\CC^N)}\\
		&=\int_{\Omega_p} \big(|\nabla \Tilde f_p|^2 +(m^2+1)|\Tilde f_p|^2\big)\,\dd x+\int_{\partial\Omega_p} \Big(m+\dfrac{H_p}{2}\Big) |\Tilde f_p|^2\dd \cH^{n-1}\\
		& \ge \int_{\Omega_p} \big(|\nabla \Tilde f_p|^2 +|\Tilde f_p|^2\big)\,\dd x=\|\Tilde f_p\|^2_{H^1(\Omega_p,\CC^N)},
	\end{align*}
	and the rest of the proof remains literally the same, which provides the self-adjointness of $A$ for such $\Omega$. Simple geometric considerations show that for $n=2$ these observations apply, for example, to curvilinear polygons without concave corners.
\end{remark}

\section{Approximations}\label{approx}

In this section we establish several approximation results for Sobolev spaces on piecewise smooth domains. The constructions
are inspired by \cite{CDN}. We need first a simple result on the approximation of non-smooth domains by smooth ones.

\begin{lemma}[``Rounding the corners'']\label{lem-corner}
	Let $\Omega\subset\RR^n$ be a bounded convex domain, then for any open neighborhood $V_0\subset\RR^n$ of its singular boundary $\partial_0\Omega$ one can find a bounded domain $\Omega_0\subset\RR^n$ with $C^\infty$-smooth boundary such that $\Omega$ and $\Omega_0$ coincide outside $V_0$.
\end{lemma}

\begin{proof}
	Without loss of generality assume that $0\in\Omega$, then due to convexity we can describe $\Omega$ using polar coordinates,
	\[
	\Omega=\big\{r\theta:\, \theta\in\SSS^{n-1},\, 0\le r< R(\theta)\big\},\quad
	\partial\Omega=\big\{R(\theta)\theta:\, \theta\in\SSS^{n-1}\big\},
	\]
	with a Lipschitz function $R:\SSS^{n-1}\to(0,\infty)$ which is $C^\infty$-smooth on
	the open set
	\[
	S_\infty:=\big\{\theta\in\SSS^{n-1}:\, R(\theta)\theta\in \partial_\infty\Omega\big\}\subset\SSS^{n-1}.
	\]
Remark $\partial_0\Omega$ is compact and that the Lipschitz map
\[
\Pi:\ \RR^n\setminus\{0\}\ni x\mapsto \dfrac{x}{|x|}\in\SSS^{n-1}
\]
is the inverse of $\theta\mapsto R(\theta)\theta$, which means that
$S_0:=\SSS^{n-1}\setminus S_\infty\equiv \Pi(\partial_0\Omega)$
is a compact subset of the open set $\Theta_0:=\Pi(V_0)\subset\SSS^{n-1}$, and
$\big\{R(\theta)\theta:\, \theta\in S_0\big\}=\partial_0\Omega\subset V_0$.

We now choose two open neighborhoods $\Theta_1,\Theta_2$ of $S_0$ in $\SSS^{n-1}$
such that $\overline{\Theta_2}\subset \Theta_1$ and $\overline{\Theta_1}\subset \Theta_0$.
Due to compactness of $\overline{\Theta_1}$ there is an $\eps>0$ such that
\[
\big\{r\theta: \, \theta\in \overline{\Theta_1},\, \big|r-R(\theta)\big|\le\eps\big\}\subset V_0,
\]	
and we can choose a $C^\infty$-function $R_1:\Theta_1\to(0,\infty)$ such that $\big|R_1(\theta)-R(\theta)\big|\le\eps$ for all $\theta\in \Theta_1$.
Now choose a cut-off function $\chi\in C^\infty(\SSS^{n-1})$ with $0\le \chi\le 1$ such that $\supp\chi\subset \Theta_1$
and $\chi=1$ on $\Theta_2$, then
$R_0:=(1-\chi)R+\chi R_1:\SSS^{n-1}\to(0,\infty)$ is smooth, and the graphs of $\theta\mapsto R(\theta)\theta$ and $\theta\mapsto R_0(\theta)\theta$ coincide outside $V_0$. It follows that
the domain
\[
\Omega_0:=\big\{r\theta:\, \theta\in\SSS^{n-1},\, 0\le r< R_0(\theta)\big\}
\]
satisfies all the requirements.	
\end{proof}

The following assertion is \cite[Lem.~2.4]{CDN} with a reformulation adapted to the subsequent use:
\begin{lemma}\label{lem-cdn}
	Let $\chi\in C^\infty_c(\RR)$ with  $0\le\chi\le 1$ such that $\supp\chi\in(-1,1)$ and $\chi=1$ in an open neighborhood of $0$.	For $\alpha\in(0,1)$ and $n\in\NN$ define functions
	\begin{gather*}
	\rho_{\alpha,n}:\ [0,\infty)\ni t \mapsto \min\{t^\alpha,1\}\big(1-\chi(n t)\big)\in\RR.
	\end{gather*}
	Then one can choose $\alpha_j\xrightarrow{j\to\infty} 0$ and $n_j\xrightarrow{j\to\infty}\infty$ such that the functions
	\[
	\psi_j:\ \RR^2\ni x\mapsto \rho_{\alpha_j,n_j}\big(|x|\big)\in\RR
	\]
	satisfy\textup{:}
	\begin{itemize}
		\item[(a)] for any $j$ one has $0\le \psi_j\le 1$, 
		the function $\psi_j$ vanishes in an open neighborhood of $0$ in $\RR^2$,
		and $\psi_j(x)\xrightarrow{j\to\infty}1$ for all $x\ne 0$,
		\item[(b)] $\|\psi_j f-f\|_{H^1(\RR^2)}\xrightarrow{j\to\infty}0$ for any $f\in H^1(\RR^2)$,	
		\item[(c)] there is a constant $C>0$ such that $\|\psi_j f\|_{H^1(\RR^2)}\le C\|f\|_{H^1(\RR^2)}$ for any $f\in H^1(\RR^2)$ and any $j\in\NN$.
	\end{itemize} 
\end{lemma}
 Remark that \cite{CDN} only states (b) explicitly. The part (a) follows from the explicit structure of $\psi_j$, and (c) follows from (b) in virtue of Banach-Steinhaus theorem.

\begin{lemma}[Cut-off near a submanifold]\label{lem-subman}
	Let $n\ge 2$ and  $\Gamma$ be a compact $(n-2)$-dimensional submanifold in $\RR^n$.
	For $x\in\RR^n$ let $d(x)$ be the distance from $x$ to $\Gamma$. Then the functions
	$\varphi_j:=\psi_j\circ d$ with $\psi_j$ from Lemma~\ref{lem-cdn}
	satisfy
	\begin{equation}
		\label{psi0}
	\|\varphi_j f-f\|_{H^1(\RR^n)}\xrightarrow{j\to\infty}0 \text{ for any } f\in H^1(\RR^n).
	\end{equation}
	In addition, $\varphi_j=0$ in an open neighborhood of $\Gamma$,
	with $0\le \varphi_j\le 1$, and $\varphi_j(x)\xrightarrow{j\to\infty}1$
	for any $x\in\RR^n\setminus\Gamma$.
\end{lemma}

\begin{proof}
For $r>0$ denote $B_2(0,r):=\{x\in\RR^2:\, |x|<r\}$. It is a classical result of differential geometry that for  some sufficiently small $\eps>0$ the $\eps$-neighborhood $\Gamma_\eps$ of $\Gamma$ admits a tubular neighborhood, see e.g. \cite[Thm.~10.19]{lee}. In the coordinate form this means that one can cover $\Gamma_\eps$ by finitely many open sets $V_1,\dots,V_L$ such that 
for any $s\in \Gamma_\ell:=V_\ell\cap\Gamma$ one can choose an orthonormal basis $\big(n_1(s),n_2(s)\big)$ in the normal space to $\Gamma$ at $s$ depending smoothly on $s$ and the maps
\[
\Phi_\ell: \Gamma_\ell\times B_2(0,\eps)\ni (s,t_1,t_2)\mapsto s+t_1 n_1(s)+t_2 n_2(s)\in V_\ell,
\]
are diffeomorphisms, with $d\big(\Phi_\ell(s,t_1,t_2)\big)=\sqrt{t_1^2+t^2_2}=|t|$, and the map $f\mapsto f\circ \Phi_\ell$
defines an isomorphism between $H^1(V_\ell)$ and $H^1\big(\Gamma_\ell\times B_2(0,\eps)\big)$.
Without loss of generality we assume that $\eps\in (0,1)$.

Let $f\in H^1(\RR^n)$. Note first that by construction one has $\varphi_j f=f$ outside $\Gamma_1$.
Furthermore,
for $x\in\Gamma_1\setminus\Gamma_\eps$ and all sufficiently large $j$ one has $(\varphi_j f)(x)= d(x)^{\alpha_j}f(x)$, and using $\alpha_j\to 0$ and the dominated convergence theorem we arrive at
$\|\varphi_j f-f\|_{L^2(\Gamma_1\setminus\Gamma_\eps)}\xrightarrow{j\to\infty} 0$. In addition, for the same $x$ one has
\begin{gather*}
\nabla (\varphi_j f)(x)-\nabla f(x)=
d(x)^{\alpha_j} \nabla f(x)-\nabla f(x)
+\alpha_j \dfrac{\nabla d(x)}{d(x)^{1-\alpha_j}} f(x),\\
\int_{\Gamma_1\setminus\Gamma_\eps} \big|\nabla (\varphi_j f)(x)-\nabla f\big|^2\dd x
\le
2\int_{\Gamma_1\setminus\Gamma_\eps} |d^{\alpha_j} \nabla f-\nabla f|^2\dd x
+2\alpha_j^2 \int_{\Gamma_1\setminus\Gamma_\eps}
\Big|\dfrac{\nabla d}{d^{1-\alpha_j}}\Big|^2|f|^2\dd x.
\end{gather*}
The first summand on the right-hand side of the last inequality
converges to zero by the dominated convergence theorem, while for the second summand we have
\[
2\alpha_j^2 \int_{\Gamma_1\setminus\Gamma_\eps}
\Big|\dfrac{\nabla d}{d^{1-\alpha_j}}\Big|^2|f|^2\dd x
\le \dfrac{2\alpha_j^2}{\eps^{2(1-\alpha_j)}}
\int_{\Gamma_1\setminus\Gamma_\eps}|f|^2\dd x\xrightarrow{j\to\infty}0.
\]

The above considerations show that $\|\varphi_j f-f\|_{H^1(\RR^n\setminus \overline{\Gamma_\eps})}\xrightarrow{j\to\infty}0$, and it remains to check $\|\varphi_j f-f\|_{H^1(\Gamma_\eps)}\xrightarrow{j\to\infty}0$. In view of the preceding discussion
it is sufficient to show that for each $\ell\in\{1,\dots,L\}$ one has
\[
\Hat f_j\xrightarrow{j\to\infty}\Hat f \text{ in }H^1\big(\Gamma_\ell\times B_2(0,\eps)\big),
\quad  
\Hat f_j:=(\varphi_j f)\circ \Phi_\ell,\quad \Hat f:=f\circ \Phi_\ell.
\]
One has $\Hat f_j (s,t)=\psi_j(t)\Hat f(s,t)$, in particular, $\nabla_s\Hat f_j=\nabla_s\Hat f$,  and
\begin{multline*}
\|\Hat f_j-\Hat f\|^2_{H^1(\Gamma_\ell\times B_2(0,\eps_0))}=\int_{\Gamma_\ell}
\Big(
\big\|\psi_j f(s,\cdot)-f(s,\cdot)\big\|^2_{H^1(B_2(0,\eps))}\\
+\|\psi_j \nabla_s f(s,\cdot)-\nabla_s f(s,\cdot)\|^2_{L^2(B_2(0,\eps))}\Big)\dd \cH^{n-2}(s).
\end{multline*}
The subintegral expression admits integrable upper bound
and converges to $0$ for a.e. $s\in \Gamma_\ell$ as $j\to\infty$  by Lemma~\ref{lem-cdn},
so the whole integral  converges to $0$ for $j\to\infty$ by the dominated convergence theorem. This completes the proof of \eqref{psi0}, and the remaining claims on $\varphi_j$ follow from the  properties of $\psi_j$ in Lemma \ref{lem-cdn}.
\end{proof}

\begin{corollary}\label{cor-rhoj}
Let $\Omega\subset\RR^n$ be a bounded domain with piecewise smooth boundary. Then one can find functions $\rho_j\in C^\infty(\RR^n)$ with $j\in\NN$ such that
\[
0\le\rho_j\le 1,\quad \rho_j=0 \text{ in an open neighborhood of }\partial_0\Omega,
\quad 
\rho_j(x)\xrightarrow{j\to\infty}1 \text{ for a.e. }x\in \RR^n
\]
and $\rho_j f\xrightarrow{j\to\infty}f$ in $H^1(\Omega)$ for any  $f\in H^1(\Omega)$.
\end{corollary}

\begin{proof}
Let $\Gamma_1,\dots,\Gamma_K$ be compact $(n-2)$-dimensional submanifolds whose union contains $\partial_0\Omega$ and denote by $d_k(x)$ the distance from $x\in\RR^n$ to $\Gamma_k$.

Let $f\in H^1(\Omega)$ and $j\in\NN$.
Denote $h_{0}:=f$, then for each $k\in\{1,\dots,K\}$ successively use Lemma~\ref{lem-subman} 
to choose $j_k\ge j$ such that the function $h_k:=(\psi_{j_k}\circ d_k) h_{k-1}$ satisfies \[
\|h_k-h_{k-1}\|_{H^1(\Omega)}<\frac{1}{jK}.\]
By construction one has 
\begin{gather*}
	h_K= \rho_j f,\qquad
\rho_j:=\prod_{k=1}^K (\psi_{j_k}\circ d_k),\\
\|\rho_j f-f\|_{H^1(\Omega)}\equiv\|h_K-h_0\|_{H^1(\Omega,\CC^N)}
\le \sum_{k=1}^K \|h_k-h_{k-1}\|_{H^1(\Omega,\CC^N)}< K\cdot \dfrac{1}{jK}=\dfrac{1}{j},
\end{gather*}
hence, $\rho_j f\xrightarrow{j\to\infty} f$ in $H^1(\Omega)$. 
The function $\psi_{j_k}\circ d_k$ vanishes in an open neighborhood  of $\Gamma_k$, so $\rho_j$ vanishes in the union of these neighborhoods (which is also an open neighborhood of $\partial_0\Omega$). For $j\to\infty$ one has $j_k\to\infty$ for each $k$, hence, $\psi_{j_k}\circ d_k (x)\to 1$ for $x\notin\Gamma_k$, which implies $\rho_j(x)\xrightarrow{j\to\infty}1$ for any $x\notin \bigcup_{k=1}^K \Gamma_k$. Finally, the required bound $0\le\rho_j\le 1$ follows from $0\le \psi_j\le 1$ (Lemma \ref{lem-cdn}).
\end{proof}

\begin{corollary}\label{cor-dens1}
Let $\Omega\subset\RR^n$ be a bounded convex domain with piecewise smooth boundary and $m\in\RR$, then the Dirac operator $A^\Omega_m$ is essentially self-adjoint on
	\[
		\dom_0 A^\Omega_m:=\big\{f\in \dom A^\Omega_m:\ 
		f=0 \text{ in an open neighborhood of }\partial_0\Omega
		\big\}.
	\]
\end{corollary}

\begin{proof}
Let $\rho_j$ be	 as in Corollary \ref{cor-rhoj} and $f\in\dom A^\Omega_m$. As the boundary condition for $A^\Omega_m$
is invariant under multiplication by smooth scalar functions, one has $\rho_j f\in \dom_0 A^\Omega_m$ with 
$\rho_j f\xrightarrow{j\to\infty} f$ in $H^1(\Omega,\CC^N)$, and the conclusion follows by Lemma \ref{lem-equiv}.
\end{proof}

\begin{lemma}[Essential self-adjointness on smooth functions]\label{thm-ess}
	Let $\Omega\subset\RR^n$ be a bounded convex domain with piecewise smooth boundary, then for any $m\in\RR$ the Dirac operator $A^\Omega_m$ is essentially self-adjoint on
	\begin{align*}
	\dom_\infty A^\Omega_m:=\big\{f\in \dom A^\Omega_m:\ & f\in C^\infty(\overline\Omega,\CC^N), 
	\ 	f=0 \text{ in an open neighborhood of }\partial_0\Omega
	\big\}.
	\end{align*}
\end{lemma}	

\begin{proof}
Let $f\in \dom A^\Omega_m$ and $\eps>0$. By Corollary \ref{cor-dens1}
one can find a function $h\in \dom A^\Omega_m$
which vanishes an open neighborhood $U$ of $\partial_0\Omega$
and satisfies $\|h-f\|_{H^1(\Omega,\CC^N)}<\frac{\eps}{3}$.
Choose bounded open neighborhoods $V_0$ and $V$ of $\partial_0\Omega$ such that $\overline{V_0}\subset V$ and $\overline{V}\subset U$, then pick
a function $\chi\in C^\infty_c(\RR^n)$ with $0\le\chi\le 1$ such that $\supp\chi\subset V$ and $\chi=1$ in $V_0$. The map
\[
H^{\frac{1}{2}}(\partial\Omega,\CC^N)\ni\varphi\mapsto(1+\rmi\beta\alpha\cdot\nu)(1-\chi)\varphi\in H^{\frac{1}{2}}(\partial\Omega,\CC^N)
\]
is a bounded linear operator (being a multiplication by a smooth matrix function compactly supported on the regular boundary), and we denote its norm by $C_0$. Furthermore, it maps the subspace
\[
\big\{\varphi\in H^{\frac{1}{2}}(\partial \Omega,\CC^N):\  \supp\varphi\subset \partial_\infty \Omega,\ \varphi\in C^\infty(\partial_\infty\Omega,\CC^N)
\big\}
\]
in itself and, in addition,  $(1+\rmi\beta\alpha\cdot\nu)(1-\chi)\varphi=(1+\rmi\beta\alpha\cdot\nu)\varphi$
for $\supp\varphi\subset\partial\Omega\setminus V_0$.

Now let $\Omega_0$ be a bounded $C^\infty$-smooth domain which coincides with $\Omega$ outside $V_0$ (see Lemma~\ref{lem-corner}).
For any $\varphi\in H^\frac{1}{2}(\partial\Omega_0)$ there is a unique  $F\in H^1(\Omega_0)$ with 
$\Delta F=0 $ in $\Omega_0$ and $F=\varphi$  on $\partial\Omega_0$,
and the associated Poisson operator $\cE_0: H^\frac{1}{2}(\partial\Omega_0)\ni\varphi\mapsto F\in H^1(\Omega_0)$
is bounded with $\cE_0\big(C^\infty(\partial\Omega_0)\big)\subset C^\infty(\overline{\Omega_0})$,
see~\cite[Theorem~11.14]{grubb}. Remark that for any function $\varphi$ defined on $\partial \Omega$ we can consider $(1-\chi)\varphi$ as a function defined on $\partial \Omega_0$ and vanishing on $V_0\cap \partial\Omega_0$. Similarly, for any function $v_0$ on $\Omega_0$
one can consider $(1-\chi)v_0$ as a function on $\Omega$ which vanishes in $V_0\cap \Omega$.
With these identifications we conclude that
\[
\cE:\ H^{\frac{1}{2}}(\partial\Omega,\CC^N)\ni \varphi \mapsto (1-\chi)(\cE_0\otimes I_N) \big((1-\chi)\varphi\big)\in H^1(\Omega,\CC^N)
\]
is a bounded operator, and we denote by $C_1$ its norm. Moreover, it maps $C^\infty(\partial\Omega)$ to $C^\infty(\overline{\Omega})$ by construction, and 
$\cE\varphi=\varphi$ on $\partial \Omega$ for any $\varphi\in C^\infty(\partial\Omega)$ with $\varphi=0$ on $V\cap \partial\Omega$.

Let $\Tilde h\in H^1(\RR^n,\CC^N)$ be an arbitrary continuation of the above $h$. Consider its convolutions $\Tilde h_\eps:=\Tilde h \ast \rho_\eps$
with standard mollifiers $\rho_\eps$, and let $h_\eps$ be their restrictions to $\Omega$, then $h_\eps\in C^\infty(\overline{\Omega},\CC^N)$. Having in mind that $h=0$ in $U$ and the support of $h_\eps$ is contained in the $\eps$-neighborhood
of the support of $h$ we conclude that we can choose $\eps$ sufficiently small such that:
the function $g:=h_\eps\in C^\infty(\overline{\Omega},\CC^N)$ satisfies $g=0$ on $V\cap \Omega$ with
\[
\|g-h\|_{H^1(\Omega,\CC^N)}<\dfrac{\eps}{3},\qquad
\|g-h\|_{H^\frac{1}{2}(\partial \Omega)}<\dfrac{\eps}{3C_0 C_1}.
\]
Consider the function $v:=g-\cE (1+\rmi\beta\alpha\cdot\nu)(g|_{\partial\Omega})$. Due to the above considerations we have $v\in C^\infty(\overline{\Omega},\CC^N)$ with $v=0$ in $V_0$,
and $(1+\rmi\beta\alpha\cdot\nu)(v|_{\partial\Omega})=0$, i.e. $v\in \dom_\infty A$.
In addition, using $(1+\rmi\beta\alpha\cdot\nu)(h|_{\partial\Omega})=0$ we obtain
\begin{multline*}
	\|v-g\|_{H^1(\Omega,\CC^N)}=\Big\|\cE (1+\rmi\beta\alpha\cdot\nu)(g|_{\partial\Omega})\Big\|_{H^1(\Omega,\CC^N)}\le C_1 \Big\| (1+\rmi\beta\alpha\cdot\nu)(g|_{\partial\Omega})\Big\|_{H^{\frac{1}{2}}(\partial\Omega,\CC^N)}\\
	=C_1 \Big\| (1+\rmi\beta\alpha\cdot\nu)(1-\chi)(g|_{\partial\Omega}-h|_{\partial\Omega})\Big\|_{H^{\frac{1}{2}}(\partial\Omega,\CC^N)}\le C_1 C_0 \|g-h\|_{H^{\frac{1}{2}}(\partial\Omega,\CC^N)}<\dfrac{\eps}{3},
\end{multline*}	
therefore,
\begin{align*}
\|v-f\|_{H^1(\Omega,\CC^N)}&\le \|v-g\|_{H^1(\Omega,\CC^N)}+\|g-h\|_{H^1(\Omega,\CC^N)}+\|h-f\|_{H^1(\Omega,\CC^N)}<\dfrac{\eps}{3}+\dfrac{\eps}{3}+\dfrac{\eps}{3}=\eps,
\end{align*}
and the claims follows due to the arbitrariness of $\eps$.
\end{proof}

\section{Proof of Theorem~\ref{thm2} (quadratic form)}\label{proof-thm2}

As a simple application of the preceding approximations we show that the formulas for the quadratic forms of $A_m^2$
and $B_{m,M}^2$ previously obtained  for smooth $\Omega$ are also valid for convex piecewise smooth $\Omega$.
The following assertion completes the proof of Theorem~\ref{thm2}:

\begin{lemma}
	Let $\Omega\subset\RR^n$ be a bounded convex domain with piecewise smooth boundary, then
	for any $f\in \dom A^\Omega_m$ one has
	\[
	\|A^\Omega_m f\|^2_{L^2(\Omega,\CC^N)}=\int_\Omega \Big( |\nabla f|^2 +m^2|f|^2\big)\,\dd x+\int_{\partial\Omega}\Big(m+\dfrac{H}{2}\Big) |f|^2\dd \cH^{n-1},
	\]
	where $H$ is the mean curvature on $\partial_\infty\Omega$.	
\end{lemma}

\begin{proof}
	Let $f\in \dom_0 A^\Omega_m$ (see Lemma~\ref{thm-ess}) vanish in an open neighborhood  $V_0$ of $\partial_0\Omega$. By Lemma~\ref{lem-corner} one can find a bounded $C^\infty$-smooth domain $\Omega_0$ coinciding with $\Omega$ outside $V_0$. Define $f_0: \Omega_0\to\CC^N$ by
	$f_0:=f$ in $\Omega_0\setminus V_0\equiv \Omega\setminus V_0$ and $f_0:=0$ in $\Omega_0\cap V_0$,
	then $f_0\in \dom A^{\Omega_0}_m$ with
	\begin{gather*}
		\nabla f_0=\nabla f \text{ and }
	A^{\Omega_0}_m f_0=A^{\Omega}_m f \text{ in }\Omega_0\setminus V_0\equiv \Omega\setminus V_0,\\
	\nabla f_0=0 \text{ and }A^{\Omega_0}_m f_0=0 \text{ in }\Omega_0\cap V_0,
	\qquad
	\nabla f=0 \text{ and }
	A^{\Omega}_m f=0 \text{ in }\Omega\cap V_0,\\
	f_0=f \text{ on }\partial\Omega_0\setminus V_0\equiv \partial\Omega\setminus V_0,
	\quad
	f_0=0 \text{ on }\partial\Omega_0\cap V,\quad f=0 \text{ on }\partial \Omega\cap V.
	\end{gather*}
	In addition, the mean curvature $H_0$ on $\partial\Omega_0$ satisfies
	$H_0=H$ on $\partial\Omega_0\setminus V_0\equiv \partial\Omega\setminus V_0$.
	Hence,
	\begin{align*}
	\|A^\Omega_m f\|^2_{L^2(\Omega,\CC^N)}&=	\|A^\Omega_m f\|^2_{L^2(\Omega\setminus V_0,\CC^N)}=\|A^{\Omega_0}_m f_0\|^2_{L^2(\Omega_0\setminus V_0,\CC^N)}=\|A^{\Omega_0}_m f_0\|^2_{L^2(\Omega_0,\CC^N)}\\
	\text{(use Lemma~\ref{lem2}) }&=
	\int_{\Omega_0} \Big( |\nabla f_0|^2 +m^2|f_0|^2\Big)\,\dd x+\int_{\partial\Omega_0}\Big(m+\dfrac{H_0}{2}\Big) |f_0|^2\,\dd \cH^{n-1}\\
	&=\int_{\Omega_0\setminus V_0} \Big( |\nabla f_0|^2 +m^2|f_0|^2\Big)\,\dd x+\int_{\partial\Omega_0\setminus V_0}\Big(m+\dfrac{H_0}{2}\Big) |f_0|^2\,\dd \cH^{n-1}\\
	&=\int_{\Omega\setminus V_0} \Big( |\nabla f|^2 +m^2|f|^2\Big)\,\dd x+\int_{\partial\Omega\setminus V_0}\Big(m+\dfrac{H}{2}\Big) |f|^2\,\dd \cH^{n-1}\\
	&=\int_\Omega \Big( |\nabla f|^2 +m^2|f|^2\Big)\,\dd x+\int_{\partial\Omega}\Big(m+\dfrac{H}{2}\Big) |f|^2\,\dd \cH^{n-1}.
	\end{align*}	
Hence the required formula holds for all functions in $\dom_0 A^\Omega_m$. 
Remark that $H$ can become unbounded near $\partial_0\Omega$, so some  attention is needed when extending to the whole domain.

Let $f\in \dom A^\Omega_m$ and $\rho_j$ be as in Corollary \ref{cor-dens1},  then for $f_j:=\rho_j f\in \dom_0 A^\Omega_m$ we have
	\begin{align*}
		\|A^\Omega_m f\|^2_{L^2(\Omega,\CC^N)}&=\lim_{j\to\infty}\|A^\Omega_m f_j\|^2_{L^2(\Omega,\CC^N)}\\
		&=\lim_{j\to\infty} \bigg [\int_\Omega \Big( |\nabla f_j|^2 +m^2|f_j|^2\Big)\,\dd x+\int_{\partial\Omega}\Big(m+\dfrac{H}{2}\Big) |f_j|^2\,\dd \cH^{n-1}\bigg].
	\end{align*}
By Corollary \ref{cor-dens1} we have
\begin{gather*}
	\int_\Omega \Big( |\nabla f_j|^2 +m^2|f_j|^2\Big)\,\dd x\xrightarrow{j\to\infty}
	\int_\Omega \Big( |\nabla f|^2 +m^2|f|^2\Big)\,\dd x,\\
	\int_{\partial\Omega}m|f_j|^2 \,\dd \cH^{n-1}
	\xrightarrow{j\to\infty}
	\int_{\partial\Omega}m|f|^2 \,\dd \cH^{n-1}.
\end{gather*}
Furthermore, $|f_j|\le |f|$ with $f_j(x)\xrightarrow{j\to\infty}f(x)$
for a.e. $x$, and Fatou's lemma yields
\[
\int_{\partial\Omega}\dfrac{H}{2} |f_j|^2\,\dd \cH^{n-1}
\xrightarrow{j\to\infty}
\int_{\partial\Omega}\dfrac{H}{2} |f|^2\,\dd \cH^{n-1},
\]	
which concludes the proof.	
\end{proof}

As a preparation for the next section we use similar ideas to show an analogous result for the Dirac operator  $B^\Omega_{m,M}$ on $\RR^n$ defined in \eqref{bmm}.
	\begin{lemma}\label{thm17}
		Let $\Omega\subset\RR^n$ be a bounded convex domain with piecewise smooth boundary and $m,M\in\RR$, then
		for all $f\in \dom B^\Omega_{m,M}$ it holds that
		\begin{multline}
			\| B^\Omega_{m,M} f\|^2_{L^2(\RR^n,\CC^N)}
			=\int_{\Omega} \big(|\nabla f|^2 +m^2|f|^2\big)\dd x
			+ \int_{\Omega^\cc} \big(|\nabla f|^2 +M^2|f|^2\big)\dd x\\
			+(M-m)\int_{\partial\Omega} \big(| \cP_- f|^2-| \cP_+ f|^2\big)\,\dd \cH^{n-1}, \label{bmm-qf}
		\end{multline}
		where $\cP_\pm(s):=\dfrac{1 \mp \rmi\beta\alpha\cdot\nu(s)}{2}$ for $s\in\partial_\infty\Omega$.
	\end{lemma}
	
	\begin{proof}
	Recall that the formula is already proved for smooth $\Omega$ in \cite[Lemma 2.3]{mobp}.
	Assume first that $f\in H^1(\RR^n,\CC^N)$ and vanishes in an open neighbhood $V_0$ of $\partial_0\Omega$.
	By Lemma~\ref{lem-corner} one can find a bounded $C^\infty$-smooth domain $\Omega_0$ which coincides with $\Omega$ outside $V_0$. Let $\nu_0$ be the outer unit normal on $\partial \Omega_0$
	and consider the respective pointwise projectors
	\[
	\cP^0_\pm(s):=\dfrac{1 \mp \rmi\beta\alpha\cdot\nu_0(s)}{2},\quad s\in \partial \Omega_0,
	\]
	then $\cP^0_\pm f=\cP_\pm f$ on $\partial \Omega\setminus V_0$.
	Due to $B^{\Omega}_{m,M}f=B^{\Omega_0}_{m,M}f$ one can use the
	already known formula for the smooth $\Omega_0$ as follows:
	\begin{align*}
		\| B^\Omega_{m,M} f\|^2_{L^2(\RR^n,\CC^N)}&=
		\| B^{\Omega_0}_{m,M} f\|^2_{L^2(\RR^n,\CC^N)}\\
		&=\int_{\Omega_0} \big(|\nabla f|^2 +m^2|f|^2\big)\dd x
		+ \int_{\Omega_0^\cc} \big(|\nabla f|^2 +M^2|f|^2\big)\dd x\\ &\qquad+(M-m)\int_{\partial\Omega_0} \big(| \cP^0_- f|^2- |\cP^0_+ f|^2\big)\dd \cH^{n-1}\\
		&=\int_{\Omega_0\setminus V_0} \big(|\nabla f|^2 +m^2|f|^2\big)\dd x
		+ \int_{\Omega_0^\cc\setminus V_0} \big(|\nabla f|^2 +M^2|f|^2\big)\dd x\\
		&\qquad+(M-m)\int_{\partial\Omega_0\setminus V_0} \big(| \cP^0_- f|^2- |\cP^0_+ f|^2\big)\dd \cH^{n-1}\\
		&=\int_{\Omega\setminus V_0} \big(|\nabla f|^2 +m^2|f|^2\big)\dd x
		+ \int_{\Omega^\cc\setminus V_0} \big(|\nabla f|^2 +M^2|f|^2\big)\dd x\\
		&\qquad+(M-m)\int_{\partial\Omega\setminus V_0} \big(| \cP_- f|^2- |\cP_+ f|^2\big)\dd \cH^{n-1}\\
		&=\int_{\Omega} \big(|\nabla f|^2 +m^2|f|^2\big)\dd x
		+ \int_{\Omega^\cc} \big(|\nabla f|^2 +M^2|f|^2\big)\dd x\\
		&\qquad+(M-m)\int_{\partial\Omega} \big(| \cP_- f|^2-| \cP_+ f|^2\big)\,\dd \cH^{n-1}.
	\end{align*}	
		Hence, the required formula holds for the functions vanishing near $\partial_0\Omega$. By Lemma \ref{lem-cdn} the operator $B^\Omega_{m,M}$ 
		is essentially self-adjoint on the set of all such functions, which allows one to extend the formula by density to the whole of $\dom B^\Omega_{m,M}$.		
	\end{proof}
	
\section{Proof of Theorem~\ref{thm-limit} (infinite mass limit)}\label{proof-thm3}	

In this section the domain $\Omega$ and the mass parameter $m$ in $\Omega$ are fixed, and we abbreviate
\[
A:=A^\Omega_m,
\quad
B_{M}:=B^\Omega_{m,M}.
\]
Recall that $\dom_\infty A$ is defined in Lemma~\ref{thm-ess}.
Le us briefly discuss some geometric ingredients appearing in the constructions below. For $s\in \partial_\infty\Omega$
the shape operator $W(s):T_s\partial_\infty\Omega\to T_s\partial_\infty\Omega$ is defined by
$W(s):=\dd |_s \nu$. Its eigenvalues $\kappa_1(s),\dots,\kappa_{n-1}(s)$ are the principal curvatures of the boundary
at $s$. One often uses the Jacobian (with $I$ being the identity operator and $t>0$)
\begin{align}
	\label{eq-jst}
J(s,t)&:=\det \big(I+ t W(s)\big)\equiv \prod_{j=1}^{n-1}\big(1+t \kappa_j(s)\big)\equiv 1+\sum_{k=1}^{n-1}H_k(s)t^k,\\
\text{with }
H_k(s)&:=\sum_{1\le j_1<\dots<j_k\le n-1} \kappa_{j_1}(s)\cdot \ldots \cdot \kappa_{j_k}(s),\quad k\in\{1,\dots,n-1\}.\nonumber
\end{align}
In particular, $H_1=\kappa_1+\dots+\kappa_{n-1}\equiv H$.

\subsection{Upper bound}

Let $j\in\NN$ and denote $E:=E_j(A^2)$. Let $\eps>0$, then by Lemma \ref{thm-ess} and min-max principle one can find a $j$-dimensional subspace $F\subset \dom_\infty A$ such that
\begin{equation}
	\label{eeps}
\max_{f\in F\setminus\{0\}}\dfrac{\|A f\|^2_{L^2(\Omega,\CC^N)}}{\|f\|^2_{L^2(\Omega,\CC^N)}}\le E+\eps.
\end{equation}
Due to convexity the map $S:\Omega^\cc\to \partial\Omega$ given by the rule $|x-S(x)|\le |x-y|$ for any $y\in \partial\Omega$ is a well-defined Lipschitz function.
For any $f\in F$ denote by $\Tilde f$ its extension to $\RR^n$ by
\[
\Tilde f(x):=f\big(S(x)\big)e^{-M d(x)},\quad
\quad
d(x):=\dist(x,\partial\Omega),
\quad 
x\in  \Omega^\cc,
\]
and consider the $j$-dimensional subspace
$\Tilde F:=\{\Tilde f: f\in F\}\subset H^1(\RR^n,\CC^N)$.
Our goal is to show that for some  $c>0$ and
all sufficiently large $M>0$ one has
\begin{equation}
	\label{upp1}
\max_{\Tilde f\in \Tilde F\setminus\{0\}}\dfrac{\|B_{M} \Tilde f\|^2_{L^2(\RR^n,\CC^N)}}{\|\Tilde f\|^{2\mathstrut}_{L^2(\RR^n,\CC^N)}}\le E+\eps +\dfrac{c}{M},
\end{equation}
then the min-max principle implies first
$\limsup\limits_{M\to+\infty} E_j(B_M^2)\le E+\eps$,
and due to the arbitrariness of $\eps$ one obtains the upper bound
\begin{equation}
	\label{eq-upp}
\limsup_{M\to+\infty} E_j(B_{M}^2)\le E\equiv E_j(A^2).
\end{equation}

Now let us show \eqref{upp1}. For any $\Tilde f\in\Tilde F$
one has $\cP_+\Tilde f=\Tilde f$ and $\cP_-\Tilde f=0$, hence, with $f:=\Tilde f|_{\Omega}\in F\subset \dom_\infty A$ 
one can rewrite the expression of Lemma~\ref{thm17} in the form
\begin{equation}
	\label{qform3}
	\begin{aligned}	
		\|B_{M} \Tilde f\|^2_{L^2(\RR^n,\CC^N)}&\equiv \int_{\Omega} \big(|\nabla f|^2 +m^2|f|^2\big)\dd x
		+\int_{\partial\Omega} \Big(m+\dfrac{H_1}{2}\Big)|f|^2\dd \cH^{n-1}\\
		&\quad+\int_{\Omega^\cc} \big(|\nabla \Tilde f|^2 +M^2|u|^2\big)\dd x -\int_{\partial\Omega} \Big(M+\dfrac{H_1}{2}\Big)|\Tilde f|^2\dd \cH^{n-1}\\
		&\equiv \|A f\|^2_{L^2(\Omega,\CC^N)}+R_M(\Tilde f),\\
		R_M(\Tilde f)&:=\int_{\Omega^\cc} \big(|\nabla \Tilde f|^2 +M^2|\Tilde f|^2\big)\dd x -\int_{\partial\Omega} \Big(M+\dfrac{H_1}{2}\Big)|\Tilde f|^2\dd \cH^{n-1}.
	\end{aligned}
\end{equation}

Let us derive an upper bound for $R_M(\Tilde f)$ using local coordinates. Remark first that one can find
a compact subset $K\subset \partial_\infty\Omega$ with $\supp \Tilde f|_{\partial\Omega}\subset K$ for all $\Tilde f\in \Tilde F$, and then $\Tilde f(x)=0$ for all $x\in\Omega^\cc$ with $S(x)\notin K$. We further choose open subsets $V_1,\dots,V_P\subset\partial_\infty\Omega$ such that
$K\subset V_1\cup\ldots\cup V_P$ and each $V_p$ is covered by a local chart
$\varphi_p: \RR^{n-1}\supset U_p\to V_p$,
and we pick $\chi_p\in C^\infty_c(V_p)$ such that $\chi_1+\dots+\chi_P=1$ on $K$.
Then $\chi_1\circ S+\dots+\chi_P\circ S=1$ on $\supp \Tilde f\cap \Omega^\cc$ for all $\Tilde f\in\Tilde F$.
Furthermore, the maps
\[
\Phi_p:\ U_p\times(0,\infty)\ni(s,t)\mapsto \varphi_p(s)+t\nu\big(\varphi_p(s)\big)\in\Omega^\cc
\]
are diffeomorphisms satisfying the identities $d\big(\Phi_p(s,t)\big)=t$ and $S\big(\Phi_p(s,t)\big)=\varphi_p(s)$.

Let $\Tilde f\in\Tilde F$ and denote $f_p:=\Tilde f\circ \Phi_p$, then $f_p(s,t)=f\big(\varphi_p(s)\big)e^{-Mt}$ for all $(s,t)$, and the partition of unity $\chi_p\circ S$ and a standard change of variables yield:
\begin{align*}
	R_M&(\Tilde f)=\sum_{p=1}^P \Bigg[\int_{\Omega^\cc} (\chi_p\circ S) \big(|\nabla \Tilde f|^2 +M^2|\Tilde f|^2\big)\dd x -\int_{\partial\Omega} \chi_p\Big(M+\dfrac{H_1}{2}\Big)|\Tilde f|^2\,\dd \cH^{n-1}\Bigg]\\
	&=\sum_{p=1}^P \int_{U_p} \chi_p\big(\varphi_p(s)\big) \Bigg[ \int_0^\infty\Big(\langle \nabla f_p(s,t), G_p(s,t)^{-1} \nabla f_p(s,t)\rangle +M^2\big|f_p(s,t)\big|^2\Big) J_p(s,t)\,\dd t\\
	&\qquad\qquad-\Big(M+\dfrac{H_1\big(\varphi_p(s)\big)}{2}\Big)\big|f_p(s,0)\big|^2\Bigg]\,\sqrt{\det g_p(s)}\,\dd s,
\end{align*}
with matrices
\begin{align*}
G_p(s,t)&:=\begin{pmatrix}
	\bigg(\Big\langle \big(1+t W_p(s)\big)\partial_j\varphi(s), \big(1+t W_p(s)\big)\partial_k\varphi(s)\Big\rangle\bigg)_{j,k=1}^{n-1} &0\\
	0 & 1
\end{pmatrix},\\
g_p(s)&:=\begin{pmatrix}
	\big\langle \partial_j\varphi(s), \partial_k\varphi(s)\big\rangle
\end{pmatrix}_{j,k=1}^{n-1},\quad
W_p:=W\circ\varphi_p.
\end{align*}
and Jacobians $J_p(s,t):=J\big(\varphi_p(s),t\big)$, see \eqref{eq-jst}.
As $W_p(s)$ is self-adjoint and its eigenvalues (the principal curvatures) are non-negative due to convexity, the min-max principle
implies
$G_p(s,t)\ge G_p(s,0)\equiv g_p(s)\oplus 1$ resulting in
\[
G_p(s,t)^{-1}\le G_p(s,0)^{-1}\equiv\begin{pmatrix}
	g_p(s)^{-1} & 0\\
	0 & 1
\end{pmatrix}.
\]
Hence,
\begin{align*}
	R_M(\Tilde f)&\le \sum_{p=1}^P \int_{U_p} \chi_p\big(\varphi_p(s)\big) \Bigg[ \int_0^\infty\Big( \big\langle \nabla_s f_p(s,t),g_p^{-1}(s)\nabla_s f_p(s,t)\big\rangle +\big|\partial_t f_p(s,t)\big|^2\\
	&\quad +M^2\big|f_p(s,t)\big|^2\Big) J_p(s,t)\,\dd t
	-\Big(M+\dfrac{H_1\big(\varphi_p(s)\big)}{2}\Big)\big|f_p(s,0)\big|^2\Bigg]\sqrt{\det g_p(s)}\,\dd s\\
	&=\sum_{p=1}^P \int_{U_p} \chi_p\big(\varphi_p(s)\big) \Bigg[ \int_0^\infty\Big( \big\langle \nabla_s f\big(\varphi_p(s)\big),g_p^{-1}(s)\nabla_s f\big(\varphi_p(s)\big)\big\rangle e^{-2Mt}\\ &\quad +2M^2\big|f\big(\varphi_p(s)\big)\big|^2e^{-2Mt}\Big) J_p(s,t)\dd t
	-\Big(M+\dfrac{H_1\big(\varphi_p(s)\big)}{2}\Big)\big|f\big(\varphi_p(s)\big)\big|^2\Bigg]\\
	&\qquad \times \sqrt{\det g_p(s)}\,\dd s.
\end{align*}
A simple direct computation shows that
\[
\int_0^\infty e^{-2Mt}\dd t=\dfrac{1}{2M},\quad
\int_0^\infty J_p(s,t)e^{-2Mt}\dd t=\dfrac{1}{2M}+\sum_{k=1}^{n-1} \dfrac{k! H_k\big(\varphi_p(s)\big)}{(2M)^{k+1}},
\]
and there is a $C_p>0$ such that for all $s$ with $\varphi_p(s)\in \supp \chi_p$ and sufficiently large $M$
one has
\[
\int_0^\infty J_p(s,t)e^{-2Mt}\dd t
\le \dfrac{1}{2M^2}\Big( M+\dfrac{H_1\big(\varphi_p(s)\big)}{2}\Big) +\dfrac{C_p}{M^3},
\]
which results in
\begin{multline*}
	R_M(\Tilde f)\le \sum_{p=1}^P \int_{U_p} \chi_p\big(\varphi_p(s)\big) \bigg[ \dfrac{1}{2M}\Big\langle \nabla_s f\big(\varphi_p(s)\big),g_p^{-1}(s)\nabla_s f\big(\varphi_p(s)\big)\Big\rangle \\
	 +\dfrac{2C_p}{M} \big|f\big(\varphi_p(s)\big)\big|^2\bigg]\, \sqrt{\det g_p(s)}\,\dd s
	\le \dfrac{C}{M }\|\Tilde f\|^2_{H^1(\partial\Omega,\CC^N)}
\end{multline*}
for some constant $C>0$ independent of $\Tilde f$ and $M$. In addition, one can find another constant $C_1>0$ with
$\|\Tilde f\|^2_{H^1(\partial\Omega,\CC^N)}\le C_1\|\Tilde f\|^2_{L^2(\RR^n,\CC^N)}$ for all $\Tilde f\in\Tilde F$,
as $\Tilde F$ is finite-dimensional. Hence,
\[
R_M(\Tilde f)\le \dfrac{CC_1}{M}\|\Tilde f\|^2_{L^2(\RR^n,\CC^N)} \text{ for all }\Tilde f \in \Tilde F.
\]
By inserting this estimate and \eqref{eeps} into the representation \eqref{qform3} we obtain for all $\Tilde f\in\Tilde F\setminus\{0\}$:
\begin{align*}
	\dfrac{\|B_{M} \Tilde f\|^2_{L^2(\RR^n,\CC^N)}}{\|\Tilde f\|^{2\mathstrut}_{L^2(\RR^n,\CC^N)}}&\le 
	\dfrac{(E+\eps)\|f\|^2_{L^2(\Omega,\CC^N)}+ \dfrac{CC_1}{M}\|\Tilde f\|^2_{L^2(\RR^n,\CC^N)}}{\|\Tilde f\|^{2\mathstrut}_{L^2(\RR^n,\CC^N)}}
	\le E+\eps +\dfrac{CC_1}{M},
\end{align*}
which shows the sought estimate \eqref{upp1} with $c:=CC_1$.

We make a side remark that the upper bound did not require the boundedness of $H$.

\subsection{Lower bound}

We start with a preliminary estimate:
\begin{lemma}\label{lem-low}
Let the mean curvature $H$ of $\partial_\infty\Omega$ be bounded, then for some  $c>0$ one has
	\[
	R_\gamma(f):=\int_{\Omega^\cc} \big(|\nabla f|^2+\gamma^2 |f|^2\big)\dd x-\int_{\partial\Omega} \Big(\gamma+\dfrac{H}{2}\Big)|f|^2\dd \cH^{n-1}\ge -c\int_{\Omega^\cc}|f|^2\dd x
	\]
	for all $f\in H^1(\Omega^\cc)$ and $\gamma>0$.
\end{lemma}

\begin{proof}
We adapt some constructions from \cite[Thm.~2]{conical}. Denote $\Pi:=\partial_\infty\Omega\times(0,\infty)$ and remark that the map $\Phi:\Pi\ni (s,t)\mapsto s+t\nu(s)$ is a diffeomorphism between $\Pi$ and $\Phi(\Pi)\subset\Omega^\cc$.
For any $f\in C^\infty_c(\RR^n)$ one has, using the standard change of variables,
\begin{align*}
		R_\gamma(f)&\ge \int_{\Phi(\Pi)} \big(|\nabla f|^2+\gamma^2 |f|^2\big)\dd x-\int_{\partial_\infty\Omega} \Big(\gamma+\dfrac{H}{2}\Big)|f|^2\dd \cH^{n-1}\\
		&=\int_{\partial_\infty\Omega} \int_0^\infty\bigg(\big|(\nabla f) \big(\Phi(s,t)\big)\big|^2+
		\gamma^2\big|f \big(\Phi(s,t)\big)\big|^2\bigg)J(s,t)\,\dd t\,\dd\cH^{n-1}(s)\\
	&\quad	-\gamma \int_{\partial_\infty\Omega} \Big(\gamma+\dfrac{H_1(s)}{2}\Big)\big|f\big(\Phi(s,0)\big)\big|^2\dd \cH^{n-1}(s)
\end{align*}
with $J$ given by \eqref{eq-jst}. Furthermore,
\begin{equation}
	\label{ders}
\big|(\nabla f) \big(\Phi(s,t)\big)\big|^2
\ge
\Big|\big\langle\nu(s),(\nabla f) \big(\Phi(s,t)\big)\big\rangle\Big|^2
=\Big| \dfrac{\partial}{\partial t} f\big(\Phi(s,t)\big)\Big|^2,
\end{equation}
hence,
\begin{equation}
	\label{rg00}
\begin{aligned}
	R_\gamma(f)\ge R'_\gamma(f):=\int_{\partial_\infty\Omega} \Bigg[\int_0^\infty&\Big( \big|\partial_t g(s,t)\big|^2+\gamma^2\big|g(s,t)\big|^2\Big)J(s,t)\,\dd t\\
	&-\Big(\gamma+\dfrac{H_1}{2}\Big) \big|g(s,0)\big|^2\Bigg]\dd \cH^{n-1}(s)
	\quad \text{with}\quad g:=f\circ\Phi.
\end{aligned}	
\end{equation}
We further consider the function
$h:=\sqrt{J} g$, then
\begin{align*}
|\partial_t g|^2&=\Big|\partial_t \dfrac{h}{\sqrt{J}}\Big|^2
=\Big|\dfrac{\partial_t h}{\sqrt{J}}-\dfrac{h\partial_t J}{2J\sqrt{J}}\Big|^2=\dfrac{1}{J}\Big|\partial_t h-\dfrac{\partial_t J}{2J} h\Big|^2\\
&=\dfrac{1}{J}\bigg(
|\partial_t h|^2-\dfrac{\partial_t J}{2J}\partial_t |h|^2
+\Big(\dfrac{\partial_t J}{2J}\Big)^2 |h|^2
\bigg),
\end{align*}
therefore,
\begin{multline*}
	R'_\gamma(f)=\int_{\partial_\infty\Omega} \Bigg[\int_0^\infty\Big( |\partial_t h|^2-\dfrac{\partial_t J}{2J}\partial_t |h|^2
	+\Big(\dfrac{\partial_t J}{2J}\Big)^2 |h|^2+\gamma^2|h|^2\Big)\,\dd t\\
	-\Big(\gamma+\dfrac{H_1(s)}{2}\Big) \big|h(s,0)\big|^2\Bigg]\dd \cH^{n-1}(s).
\end{multline*}
Remark that
\begin{equation}
	\label{jst1}
\begin{aligned}
	\dfrac{\partial_t J(s,t)}{2J(s,t)}&=\dfrac{1}{2}\partial_t\log J(s,t)=\dfrac{1}{2}\partial_t\sum_{j=1}^{n-1}\log\big(1+t\kappa_j(s)\big)=
	\dfrac{1}{2}\sum_{j=1}^{n-1}\dfrac{\kappa_j(s)}{1+t\kappa_j(s)},\\
	\dfrac{\partial_t J(s,t)}{2J(s,t)}&\bigg|_{t=0}=\dfrac{1}{2}H_1(s),
\end{aligned}
\end{equation}
and the integration by parts yields
\begin{align*}
\int_0^\infty\Big(-\dfrac{\partial_t J(s,t)}{2J(s,t)}\partial_t \big|h(s,t)\big|^2\Big)\,\dd t&=\bigg[-\dfrac{\partial_t J(s,t)}{2J(s,t)} \big|h(s,t)\big|^2\bigg]_{t=0}^{t=\infty}+\int_0^\infty\partial_t\Big(\dfrac{\partial_t J(s,t)}{2J(s,t)}\Big)\big|h(s,t)\big|^2\,\dd t\\
&=\dfrac{H_1(s)}{2}\big|h(s,0)\big|^2+\int_0^\infty\partial_t\Big(\dfrac{\partial_t J(s,t)}{2J(s,t)}\Big)\big|h(s,t)\big|^2\,\dd t,
\end{align*}
and the substitution into the preceding expression for $R'_\gamma(f)$ gives
\begin{equation}
	\label{rprime1}
\begin{aligned}
	R'_\gamma(f)&=\int_{\partial_\infty\Omega} \Bigg[\int_0^\infty\Big( \big|\partial_t h(s,t)\big|^2
	+\gamma^2\big|h(s,t)\big|^2\Big)\,\dd t-\gamma \big|h(s,0)\big|^2\Bigg]\dd \cH^{n-1}(s)\\
	&\quad+
	\int_{\partial_\infty\Omega} \int_0^\infty
	\bigg(\partial_t\Big(\dfrac{\partial_t J}{2J}\Big)
	+\Big(\dfrac{\partial_t J}{2J}\Big)^2\bigg) |h|^2 \,\dd t\, \dd \cH^{n-1}.
\end{aligned}
\end{equation}
Using
\begin{align*}
\big|h(s,0)\big|^2&=-\int_0^\infty \partial_t \big|h(s,t)\big|^2\,\dd t=-\int_0^\infty 2\Re \Big(\overline{\partial_t h(s,t)} h(s,t)\Big)\,\dd t\\
&\le \int_0^\infty 2 \big|\partial_t h(s,t)\big|\cdot \big| h(s,t)\big|\,\dd t
\le\int_0^\infty \bigg(\dfrac{1}{\gamma} \big|\partial_t h(s,t)\big|^2+ \gamma \big| h(s,t)\big|^2\bigg)\,\dd t,
\end{align*}
we show that the first integral in \eqref{rprime1} is always non-negative.
To estimate the second integral in \eqref{rprime1} we use \eqref{jst1}:
\begin{gather*}
\partial_t\Big(\dfrac{\partial_t J(s,t)}{2J(s,t)}\Big)=
-\dfrac{1}{2}\sum_{j=1}^{n-1}\dfrac{\kappa_j(s)^2}{\big(1+t\kappa_j(s)\big)^2}\ge -\dfrac{n-1}{2} H(s)^2,\\
\partial_t\Big(\dfrac{\partial_t J}{2J}\Big)
+\Big(\dfrac{\partial_t J}{2J}\Big)^2\ge 
\partial_t\Big(\dfrac{\partial_t J}{2J}\Big)
\ge -c,\qquad c:=\dfrac{n-1}{2}\|H\|_\infty^2.
\end{gather*}
The substitution into \eqref{rprime1} gives
\begin{align*}
R'_\gamma(f)&\ge -c\int_{\partial_\infty\Omega} \int_0^\infty
|h|^2 \,\dd t\, \dd \cH^{n-1}(s)
=-c\int_{\partial_\infty\Omega} \int_0^\infty
\Big|f\big(\Phi(s,t)\big)\Big|^2 J(s,t)\,\dd t\, \dd \cH^{n-1}(s)\\
&=-c \int_{\Phi(\Pi)} |f|^2\,\dd x\ge -c \int_{\Omega^\cc} |f|^2\,\dd x,
\end{align*}
and using \eqref{rg00} we arrive at
\[
R_\gamma(f)\ge -c \int_{\Omega^\cc} |f|^2\,\dd x \text{ for all }f\in C^\infty_c(\RR^n),
\]
which extends by density to all $f\in H^1(\Omega^\cc)$.
\end{proof}

With the lower bound of Lemma \ref{lem-low} at hand, the remaining analysis is very close to the constructions in~\cite[Sec.~5.3]{mobp}.
We use the formula \eqref{bmm-qf} for $\| B_M f\|^2_{L^2(\RR^n,\CC^N)}$ to obtain 
\begin{align*}
	\| B_{M} f\|^2_{L^2(\RR^n,\CC^N)}&=\int_{\Omega} \big(|\nabla f|^2 +m^2|f|^2\big)\dd x\\
	&\quad +\int_{\partial\Omega} \Big(m-\dfrac{1}{\sqrt{M}}+\dfrac{H}{2}\Big)|f|^2\dd \cH^{n-1} + 2(M-m)\int_{\partial\Omega} | \cP_- f|^2\dd \cH^{n-1}\\
	&\quad +\bigg[\int_{\Omega^\cc} \big(|\nabla f|^2 +M^2|f|^2\big)\dd x -\int_{\partial\Omega} \Big(M-\dfrac{1}{\sqrt{M}}+\dfrac{H}{2}\Big)|f|^2\dd \cH^{n-1}\bigg]\\
	&\ge \int_{\Omega} \big(|\nabla f|^2 +m^2|f|^2\big)\dd x+\int_{\partial\Omega} \Big(m-\dfrac{1}{\sqrt{M}}+\dfrac{H}{2}\Big)|f|^2\dd \cH^{n-1}\\
	&\qquad + 2(M-m)\int_{\partial\Omega}| \cP_- f|^2\dd \cH^{n-1}
	 +C_M\int_{\Omega^\cc}|f|^2\dd x,\quad C_M:=2\sqrt{M}-\dfrac{1}{M}-c,
\end{align*}
where the term is the square brackets was estimated from below by Lemma~\ref{lem-low}.
Using the  Neumann decoupling of $\Omega$ and $\Omega^\cc$ along $\partial\Omega$ and the min-max principle
we conclude that for any fixed $j\in\NN$ and all $M>0$ one has
\begin{equation}
	\label{low55}
E_j(B_{M}^2)\ge E_j(K_{M} \oplus C_M I^\cc),
\end{equation}
where $K_{M}$ is the self-adjoint operator with compact resolvent in $L^2(\Omega,\CC^N)$ whose sesquilinear form $k_{M}$  is defined on $\dom k_{M}:=H^1(\Omega,\CC^N)$ by
\begin{multline*}
	k_{M}(f,f)
	=\int_{\Omega} \big(|\nabla f|^2 +m^2|f|^2\big)\dd x\\
	+\int_{\partial\Omega} \Big(m-\dfrac{1}{\sqrt{M}}+\dfrac{H}{2}\Big)|f|^2\dd \cH^{n-1} + 2(M-m)\int_{\partial\Omega} | \cP_- f|^2\dd \cH^{n-1}
\end{multline*}
and $I^\cc$ is the identity operator in $L^2(\Omega^\cc,\CC^N)$. Due to the upper bound \eqref{eq-upp}
the eigenvalue $E_j(B_{M}^2)$ remains bounded for $M\to+\infty$, while $C_M\xrightarrow{M\to+\infty}+\infty$, and
 \eqref{low55} yields
\begin{equation}
	\label{low4}
\liminf_{M\to+\infty} E_j(B_{M}^2)\ge \liminf_{M\to+\infty} E_j(K_{M}).
\end{equation}
For each $f\in H^1(\Omega,\CC^N)$ the function $M\to k_{M}(f,f)$
is monotonically increasing, with
\[
\Big\{f\in \bigcap_{M>0} \dom k_{M}:\ \sup_{M>0} k_{M}(f,f)<\infty
\Big\}
=\big\{f\in H^1(\Omega,\CC^N): \ \cP_- f=0 \text{ on }\partial\Omega\big\}
=\dom A,
\]
and for any $f\in\dom A$ one has
$\lim\limits_{M\to+\infty}k_{M}(f,f)=\|A f\|^2_{L^2(\Omega,\CC^N)}$.
Hence, the monotone convergence principle implies
$\lim\limits_{M\to+\infty} E_j(K_{M})=E_j(A^2)$, see \cite{simon, weid} for the general theory or \cite[Prop.~7]{mobp} for the most adapted formulation. The substitution into \eqref{low4} gives the required lower bound
$\liminf\limits_{M\to+\infty} E_j(B_{M}^2)\ge E_j(A^2)$.

\begin{remark}\label{rmk-end}
While the boundedness of $H$ is explicitly used in Lemma~\ref{lem-low}, we believe that this assumption
is technical and can probably be avoided by applying more advanced proof methods. It is implicitly supported by the fact
that the boundary integral containing $H$ in the expression for $\|A^\Omega_m f\|^2$ in Theorem~\ref{thm2} turns out to be always finite, while $H$ can be unbounded. The proof of Lemma~\ref{lem-low}
does not exploit the tangential derivatives, as they are dropped in the step \eqref{ders}. One may expect
that these tangential derivatives can actually compensate the unbounded curvature terms (using the idea that a Schr\"odinger
operator with an unbounded negative potential may still remains lower semibounded). Nevertheless such an improvement
would require a very advanced analysis (like a control of unbounded curvatures near the singular boundary), which goes significantly beyond the scope of the present work.
 \end{remark}

\appendix

\section{Dirac matrices}\label{appa}

The most frequently used choices for the Dirac matrices $\alpha_k$ and $\beta$ entering \eqref{eqedm} are based on the Pauli matrices
\[
\sigma_1:=\begin{pmatrix}
	0 & 1\\
	1 & 0
\end{pmatrix},
\quad
\sigma_2:=\begin{pmatrix}
	0 & -\rmi\\
	\rmi & 0
\end{pmatrix},
\quad
\sigma_3:=\begin{pmatrix}
	1 & 0\\
	0 & -1
\end{pmatrix}.
\]
For $n=2$ one has $N=2$, and the standard choice is $\alpha_k:=\sigma_k$ for $k\in\{1,2\}$ and $\beta:=\sigma_3$.
For $n=3$ one has $N=4$, and one often uses
\[
\alpha_k:=\begin{pmatrix}
	0 & \sigma_k\\
	\sigma_k & 0
\end{pmatrix},
\quad k\in\{1,2,3\},
\quad
\beta:=\begin{pmatrix}
	I_2 & 0\\
	0 & -I_2
\end{pmatrix},
\]
where $I_2$ is the $2\times 2$ identity matrix. For higher dimensions
one can use various iterative procedures: see e.g. \cite[Chapter 15]{dg} or \cite[Appendix~E]{wit}
for a general theory, or \cite[Sec.~2.1]{LOB} for a more specific description. We remark that the choice of $\alpha_k$ and $\beta$
is unique up to similarity transforms and sign changes, as described in the aforementioned references,
and all associated Dirac operators are unitarily equivalent to each other.

\section*{Acknowledgments}
The work was partially supported by the Deutsche Forschungsgemeinschaft (DFG, German Research Foundation), project 491606144.

\end{document}